\documentclass[12pt]{article}

\usepackage[T1]{fontenc}
\usepackage{avant}
\usepackage[cp1250]{inputenc}
\usepackage{amsmath,amssymb,amsthm}
\usepackage{caption}
\usepackage{geometry}
\usepackage{float}

\usepackage{makeidx}
\usepackage[matrix,arrow,curve]{xy}
\usepackage{boxedminipage}
\usepackage{epic}
\usepackage{longtable}
\usepackage{ifthen}
\usepackage{tabularx}
\usepackage{parskip}
\usepackage{verbatim}
\usepackage{tikz}
\usepackage{pgfkeys}

\usepackage{dynkin-diagrams}

\title{ A solution of the problem of standard compact Clifford-Klein forms}
\author{Maciej Boche\'nski and Aleksy Tralle}

\begin{document}

\newtheorem{theorem}{Theorem}
\newtheorem{proposition}{Proposition}
\newtheorem{lemma}{Lemma}
\newtheorem{definition}{Definition}
\newtheorem{example}{Example}
\newtheorem{note}{Note}
\newtheorem{corollary}{Corollary}
\newtheorem{remark}{Remark}
\newtheorem{question}{Question}

\maketitle{}
\abstract{We solve the problem of classification of standard compact Clifford-Klein forms of homogeneous spaces of simple non-compact real Lie groups under the following assumptions:
\begin{enumerate}
\item    either standard compact Clifford-Klein forms are given by triples$(G,H,L),$ with $G,H,L$ linear connected which arise  from triples $(\mathfrak{g},\mathfrak{h},\mathfrak{l})$ of real Lie algebras with $\mathfrak{g}$ absolutely simple, and at least one of the complexified subalgebras  $\mathfrak{h}^c\subset\mathfrak{g}^c$ or $,\mathfrak{l}^c\subset\mathfrak{g}^c$, is regular, or
\item $\mathfrak{g},\mathfrak{h},\mathfrak{l}$ are  absolutely simple.
\end{enumerate}
The solution is that such triples yield standard compact Clifford-Klein forms if and only if $\mathfrak{g}=\mathfrak{h}+\mathfrak{l}$ and $\mathfrak{h}\cap\mathfrak{l}$ is compact.}

\vskip10pt

\section{Introduction}
\subsection{The problem and motivation}
 Let $G$ be a connected real semisimple linear Lie group and $H\subset G$ a closed connected reductive subgroup.
\begin{definition}[T. Kobayashi, \cite{k-k},\cite{kob}]\label{def:standard}
{\rm A homogeneous space $G/H$ admits a {\it standard compact Clifford-Klein form}, if there exists a  discrete subgroup $\Gamma\subset G$ and closed reductive subgroup $L\subset G$ containing $\Gamma$ as a co-compact lattice such that $\Gamma$ (and thus  $L$) acts properly and co-compactly on $G/H.$ }
\end{definition}
\noindent This notion was first introduced in \cite{kob}. A list of such spaces was given in \cite{kob1} and \cite{kob-y}. The term "standard" was first used in \cite{k-k}.
 Standard Clifford-Klein forms are determined by triples $(G,H,L)$, where $H$ and $L$ are reductive subgroups. The general theory \cite{kob} shows that we can restrict ourselves to considering semisimple Lie subgroups $H,L$ (see Section \ref{subs21} for details), and we will do that throughout this article. The following is well known.
\begin{theorem}\label{thm:table}
Assume that $\mathfrak{g}$ is absolutely simple, and $\mathfrak{h}$ and $\mathfrak{l}$ are reductive. All triples $(\mathfrak{g},\mathfrak{h},\mathfrak{l})$  listed in Table 1 yield standard compact Clifford-Klein forms. 
\end{theorem}

\begin{center}
 \begin{table}[h]
 \centering
 {\footnotesize
 \begin{tabular}{| c | c | c |}
   \hline
   \multicolumn{3}{|c|}{ \textbf{\textit{Standard Clifford-Klein forms}}} \\
   \hline                        
   $\mathfrak{g}$ & $\mathfrak{h}$ & $\mathfrak{l}$ \\
   \hline
		$\mathfrak{su}(2,2n)$ & $\mathfrak{sp}(1,n)$ & $\mathfrak{su}(1,2n)$ \\
		\hline
		$\mathfrak{su}(2,2n)$ & $\mathfrak{sp}(1,n)$ & $\mathfrak{u}(1,2n)$ \\
		\hline
		$\mathfrak{so}(2,2n)$ & $\mathfrak{so}(1,2n)$ & $\mathfrak{su}(1,n)$ \\
		\hline
		$\mathfrak{so}(2,2n)$ & $\mathfrak{so}(1,2n)$ & $\mathfrak{u}(1,n)$ \\
		\hline
		$\mathfrak{so}(4,4n)$ & $\mathfrak{so}(3,4n)$ & $\mathfrak{sp}(1,n)$ \\
		\hline
		$\mathfrak{so}(4,4n)$ & $\mathfrak{so}(3,4n)$ & $\mathfrak{sp}(1,n)\times\mathfrak{sp}(1)$ \\
		\hline
		$\mathfrak{so}(4,4n)$ & $\mathfrak{so}(3,4n)$ & $\mathfrak{sp}(1,n)\times\mathfrak{so}(2)$ \\
		\hline
		$\mathfrak{so}(4,4n)$ & $\mathfrak{so}(3,4n)$ & $\mathfrak{sp}(1,n)$ \\
		\hline
		$\mathfrak{so}(3,4)$ & $\mathfrak{g}_{2(2)}$ & $\mathfrak{so}(1,4)\times\mathfrak{so}(2)$ \\
		\hline
		$\mathfrak{so}(3,4)$ & $\mathfrak{g}_{2(2)}$ & $\mathfrak{so}(1,4)$ \\
		\hline
		$\mathfrak{so}(8,8)$ & $\mathfrak{so}(7,8)$ & $\mathfrak{so}(1,8)$ \\
		\hline
		$\mathfrak{so}(4,4)$ & $\mathfrak{so}(3,4)$ & $\mathfrak{so}(1,4)\times\mathfrak{so}(3)$ \\
		\hline
		$\mathfrak{so}(4,4)$ & $\mathfrak{so}(3,4)$ & $\mathfrak{so}(1,4)\times\mathfrak{so}(2)$ \\
		\hline
		$\mathfrak{so}(4,4)$ & $\mathfrak{so}(3,4)$ & $\mathfrak{so}(1,4)$ \\
		\hline
 \end{tabular}
 }
 \caption{
 }
 \label{tttab}
 \end{table}
\end{center}

It is important to note that these are {\it the only} known examples of homogeneous spaces admitting compact Clifford-Klein forms (with $\mathfrak{g}$ simple and $\mathfrak{h}$ proper and non-compact). All Lie algebra triples in Table \ref{tttab} come from \cite{on}, Table 2. In this work, Onishchik classifies all triples $(\mathfrak{g},\mathfrak{h},\mathfrak{l})$ such that $\mathfrak{g}$ is simple, $\mathfrak{h},\mathfrak{l}$ are reductive and $\mathfrak{g}=\mathfrak{h}+\mathfrak{l}$. Such triples are called Lie algebra decompositions. We omit the complete description of the embeddings $\mathfrak{h}\hookrightarrow\mathfrak{g}$, referring to \cite{on}. Foe example, in the last four lines one of the subalgebras is embedded via an appropriate spin representation (see \cite{on}, the proof of Theorem 4.1). Also, the last three lines in the table are already contained in the table as special cases with $n=1$. However, we prefer to leave it in the form given by Table 2 in \cite{on}, for the convenience of the reader. Let us also note that according to the general theory of compact Clifford-Klein forms \cite{kob},  it is allowed to interchange the roles of $\mathfrak{h}$ and $\mathfrak{l}$, and we will do that throughout this article.

In this paper we prove that under some mild restrictions Table \ref{tttab} exhausts all the possibilities for compact standard Clifford-Klein forms. The main result is as follows.

\begin{theorem}\label{thm:class} Assume that $\mathfrak{g}$ is absolutely simple, and $\mathfrak{h},\mathfrak{l}$ satisfy the following:
\begin{enumerate}
\item either the complexified embedding $\mathfrak{h}^c\hookrightarrow\mathfrak{g}^c$ is regular, or 
\item both $\mathfrak{h}$ and $\mathfrak{l}$   are proper absolutely simple subalgebras of non-compact type in $\mathfrak{g}$.
\end{enumerate}
Then  only Lie algebra triples $(\mathfrak{g},\mathfrak{h},\mathfrak{l})$ which yield compact standard Clifford-Klein forms are Lie algebra decompositions, that is $\mathfrak{g}=\mathfrak{h}+\mathfrak{l}$.
\end{theorem}
Note that if $\mathfrak{g}$ is absolutely simple exceptional then, by \cite{bjt}, for any proper reductive subalgebra $\mathfrak{h}\subset \mathfrak{g}$ of non-compact type the corresponding homogeneous space does not admit compact standard Clifford-Klein forms. Therefore, in the proof of Theorem \ref{thm:class} we may assume that $\mathfrak{g}$ is classical simple.

It is interesting  that the proof of Theorem \ref{thm:class} yields a characterization of proper co-compact actions of reductive Lie groups on homogeneous spaces $G/H$ determined by simple absolutely simple Lie groups $G$ and (non-compact) absolutely simple subgroups (we thank T. N. Venkataramana for this observation).
\begin{theorem}\label{thm:proper-Lie} Let $G/H$ be a homogeneous space determined by an absolutely simple Lie groups $G,H,$ $H\subset G$. An absolutely simple Lie subgroup $L\subset G$ acts on $G/H$ properly and co-compactly, if and only if $G=H\cdot L$ and $H\cap L$ is compact.
\end{theorem} 

There are various reasons for studying compact Clifford-Klein forms (we mention only some of them): 
\begin{enumerate}
\item spectral analysis on pseudo-Riemannian manifolds \cite{k-k}, \cite{k-k1},\cite{kob-int},

\item  a problem of T. Kobayashi: {\it does  any homogeneous space $G/H$ admitting compact Clifford-Klein form admits also a standard one?} Clearly, one of the ways to attack it is to understand all standard Clifford-Klein forms,
\item constructing various compact manifolds with particular geometric structures. For example, a homogeneous space of reductive type is an important example of a homogeneous  pseudo-Riemannian manifold. This observation was exploited in \cite{bt} to explore Einstein metrics on compact Clifford-Klein forms. The existence of some invariant geometric structures is an obstruction to compact Clifford-Klein forms, for example, none of para-Hermitian symmetric spaces admits a compact Clifford-Klein form \cite{kob-y}.    
\end{enumerate}   
Since the foundational works of Benoist, Labourie and T. Kobayashi \cite{ben},\cite{bela},\cite{kob} a lot of work has been done in the whole area. Since there are several excellent surveys of the whole topic (see \cite{kob1},\cite{kob-y}) we prefer to refer to these works instead of a detailed exposition.

In the context of this work, one should mention that there are many non-existence results for compact Clifford-Klein forms (not necessarily standard). We have already mentioned   \cite{bjt}. Topological obstructions to compact Clifford-Klein forms were found in \cite{Mo},\cite{M1}, \cite{T} together with examples. A procedure of constructing homogeneous spaces with no compact Clifford-Klein forms was obtained in \cite{bjst}. On the positive side, all compact standard Clifford-Klein of symmetric spaces were classified in \cite{tojo}. The existence  of {\it non-standard} compact Clifford-Klein forms is a deep mathematical problem of a different nature. 
For instance there are examples of (non-compact) Clifford-Klein forms which are in some sense strongly non-standard:  in \cite{bym} one can find examples of homogeneous spaces $G/H$ of reductive type which admit proper actions of discrete subgroups of $G$ isomorphic to co-compact lattices in $O(n,1),$ $n=2,3,4$ but do not admit proper actions of any non-compact semisimple subgroups of $G.$
We also refer to a very recent work \cite{MST}.

Finally, let us mention that we always consider the problem of standard Clifford-Klein forms under the assumption that $\mathfrak{h}$ and $\mathfrak{l}$ are reductive. The problem of existence of compact Clifford-Klein forms is meaningful also for other types of subgroups of $G$, for example for the solvable ones. There are no such forms for nilpotent $N\subset G$ \cite{ben} and for solvable $\Gamma\subset G$ under rather strong restrictions on $H$ \cite{bt1}. 

\subsection{Methods of proof of the main results}
The proofs in this article are rather long and sometimes technical, therefore we feel it might be useful for the reader to have some overview first. Note that by \cite{kob} when looking for  a standard compact Clifford-Klein form determined by a triple $(\mathfrak{g},\mathfrak{h},\mathfrak{l})$ one encounters with two somewhat different tasks: 
\begin{enumerate}
\item one needs to know when  a subgroup $L$ acts properly on $G/H$, and
\item  when the double coset  space $\Gamma\backslash G/H$ is compact.
\end{enumerate}
The criteria of properness were found in \cite{ben} and \cite{kob}. They are expressed in Lie algebra terms and {\it depend on the knowledge of the embeddings} $\mathfrak{h}\hookrightarrow\mathfrak{g}$ and $\mathfrak{l}\hookrightarrow\mathfrak{g}$. This is one of the main difficulties if one wants to solve the general problem, since such knowledge may be not available. The second difficulty is to learn whether the given proper action is co-compact. According to \cite{kob}  one needs to compare the non-compact dimensions of $\mathfrak{g},\mathfrak{h}$ and $\mathfrak{l}$ and prove the equality $d(\mathfrak{g})=d(\mathfrak{h})+d(\mathfrak{l})$ (still it is not an easy task, although the calculation does not depend on the embedding).  
\vskip6pt
Our methods fully use the results \cite{kob-d},\cite{kobcrit}, \cite{kob}, but:
\begin{enumerate}
\item  we derive {\it  a new property of the Iwasawa decompositions} of $\mathfrak{g}$,$\mathfrak{h}$ and $\mathfrak{l}$ of any {\it standard} triple $(\mathfrak{g},\mathfrak{h},\mathfrak{l})$. This property (called {\it well-embeddedness} )  is formulated in Lemma \ref{lemma:nh+nl} and proved in \cite{botr} with the use of \cite{abis},
\item using Lemma \ref{lemma:nh+nl} we derive some properties of the real  root space decompositions of $\mathfrak{g},\mathfrak{h}$ and $\mathfrak{l}$ and their relations with root decompositions of the complexifications $\mathfrak{g}^c$,$\mathfrak{h}^c$ and $\mathfrak{l}^c$ under the assumption that $\mathfrak{g}\not=\mathfrak{h}+\mathfrak{l}$ (this result was proved in \cite{botr}, we repeat its formulation here, Theorem \ref{thm:root-decomp}),
\item assuming now that the embedding $\mathfrak{h}^c\hookrightarrow\mathfrak{g}^c$ is regular, we derive a contradiction between Theorem \ref{thm:root-decomp} and the regularity (Theorem  \ref{thm:centr}),
\item in the non-regular case, we use Dynkin's classification of maximal subalgebras of simple classical Lie algebras and under the assumptions that $\mathfrak{h}^c$ and $\mathfrak{l}^c$ are both simple and {\it non-regular}, we are able to derive the inequality $d(\mathfrak{g})>d(\mathfrak{h})+d(\mathfrak{l})$ which completes the whole proof.
\end{enumerate}
The following remark seems to be in order. Our proofs essentially use the classification results from variuous sources: classical Dynkin's papers \cite{d},\cite{d1}, tables of fundamental irreducible representations of simple Lie algebras \cite{bou},\cite{bouu}, tables of real forms of simple complex Lie algebras,  of non-compact dimensions of simple Lie algebras  and  tables of maximal semisimple subalgebras in simple Lie algebras \cite{ov}. Parts of these used in this article are reproduced as Tables \ref{tttab3}-\ref{tab-max}  
\begin{remark} {\rm We were informed by the reviewer of our paper \cite{botr} that S. Mehdi and M. Olbrich announced in \cite{MOl} that there are no compact standard Clifford-Klein forms of simple absolutely simple Lie groups $G$ and non-compact $H,L$ except those whose Lie algebras $\mathfrak{g},\mathfrak{h}$ and $\mathfrak{l}$ are contained in Table \ref{tttab}. The proof has not been published yet.}
\end{remark}  
 
\noindent {\bf Acknowledgment}. We thank Yves Benoist, Pralay Chatterjee, Toshiyuki Kobayashi, Yosuke Morita and  T. N. Venkataramana for valuable discussions during the meeting "Zariski dense subgroups, number theory and geometric applications" held  at ICTS, Bangalore in January 2024. Also, some parts of this work were completed during  the second author research stay at the Institute of Mathematical Sciences in Chennai. His special thanks go to the Institute for a wonderful research atmosphere and hospitality.

\section{Preliminaries}

\subsection{Notation}
Our basic references are \cite{knapp},\cite{knapp1} and \cite{ov}. Assume that $(\mathfrak{g},\mathfrak{h},\mathfrak{l})$ is a standard triple, which means  that $\mathfrak{g}$ is absolutely simple, $\mathfrak{h},\mathfrak{l}$ are reductive subalgebras of $\mathfrak{g}$ and $(\mathfrak{g},\mathfrak{h},\mathfrak{l})$ induces a compact standard Clifford-Klein form.  Throughout we work with root systems of complex semisimple Lie algebras and with real root systems of their real forms. Also, we use some relations between them, which we reproduce here in order to fix notation (see \cite{ov}, Section 4.1). We may assume that there exists a Cartan involution $\theta$ of $\mathfrak{g}$ such that $\theta (\mathfrak{h}) = \mathfrak{h},$ $\theta (\mathfrak{l})=\mathfrak{l}.$ Hence the restriction of $\theta$ is a Cartan involution of $\mathfrak{h}$ and $\mathfrak{l} .$  We obtain the following Cartan decompositions
$$\mathfrak{g}=\mathfrak{k}+\mathfrak{p}, \  \mathfrak{h}=\mathfrak{k}_{h}+\mathfrak{p}_{h}, \  \mathfrak{l}=\mathfrak{k}_{l}+\mathfrak{p}_{l}, \ \ \mathfrak{k}_{h},\mathfrak{k}_{l}\subset \mathfrak{k}, \ \mathfrak{p}_{h}, \mathfrak{p}_{l}\subset \mathfrak{p}.$$
Denote by $\mathfrak{g}^c, \mathfrak{h}^c, \mathfrak{l}^c$ the complexifications of $\mathfrak{g},\mathfrak{h},\mathfrak{l},$ respectively. Let $\mathfrak{a}\subset \mathfrak{p}$ be a maximal abelian subspace of $\mathfrak{p}$ and denote by $\mathfrak{m}_{0}$ the centralizer of $\mathfrak{a}$ in $\mathfrak{k}.$ Choose a Cartan subalgebra $\mathfrak{t}$ of $\mathfrak{m}_{0}.$ Then $\mathfrak{j}:=\mathfrak{t}+\mathfrak{a}$ is a Cartan subalgebra of $\mathfrak{g}$ and the complexification $\mathfrak{j}^c$ of $\mathfrak{j}$ is a Cartan subalgebra of $\mathfrak{g}^{c}.$  Let  $\Delta$ be a root system of $\mathfrak{g}^c$ with respect to $\mathfrak{j}^c.$  Consider the root decomposition 
$$\mathfrak{g}^c=\mathfrak{j}^c+\sum_{\alpha\in \Delta} \mathfrak{g}_{\alpha}.$$
The Lie algebra $\hat{\mathfrak{j}}:=i\mathfrak{t}+\mathfrak{a}$ is a real form of $\mathfrak{j}^c$ and
$$\hat{\mathfrak{j}}= \{ A\in \mathfrak{j}^c \ | \ \alpha (A)\in \mathbb{R} \ \textrm{for any} \ \alpha\in \Delta  \} .$$
 The weights of the adjoint representation of $\mathfrak{a}$ constitute a root system $\Sigma\subset\mathfrak{a}^*$ (not necessarily reduced). There is a natural map $pr:\Delta\rightarrow\mathfrak{a}^*$ given by $\alpha\rightarrow \alpha|_{\mathfrak{a}}$. It is easy to see that $pr(\Delta)=\Sigma \cup \{ 0 \} .$ Therefore, we will call $pr(\alpha)=\alpha|_{\mathfrak{a}}\in\mathfrak{a}^*$ (if non-zero) the restricted roots. Thus,
$\Sigma = \{ \alpha|_{\mathfrak{a}} \ | \ \alpha\in\Delta   \}\setminus\{ 0\} , $
$$\mathfrak{g}_{\beta} := \sum_{\substack{\alpha\in \Delta \\ \alpha |_{\mathfrak{a}} = \beta}} \mathfrak{g}_{\alpha} \cap \  \mathfrak{g} .  $$
The reader should keep in mind, that we use the following convention. We use the same notation for complex roots in $\Delta$ and real (restricted) roots in $\Sigma$.  The meaning of these will be always clear either  from the context, or explicitly stated. 

We get the following decomposition of $\mathfrak{g}$
$$\mathfrak{g}=\mathfrak{m}_{0}+\mathfrak{a}+\sum_{\alpha \in \Sigma} \mathfrak{g}_{\alpha},$$
where $\mathfrak{m}_0=\mathfrak{z}_{\mathfrak{g}}(\mathfrak{a})$.
\noindent Let $B$ be the Killing form of $\mathfrak{g}.$ It extends onto $\mathfrak{g}^c$  as a Killing form of $\mathfrak{g}^c$ and we denote it by the same letter.  Consider the compact real form  $\mathfrak{g}_{u}:= \mathfrak{k}+ i\mathfrak{p},$  of $\mathfrak{g}^c ,$ and denote by $\tau$ the complex conjugation in $\mathfrak{g}^c$ with respect to $\mathfrak{g}_{u} .$  Also we denote by $\sigma$ the complex conjugation of $\mathfrak{g}^c$ with respect to $\mathfrak{g}$. It is well known (\cite{ov1}, p. 228) that there is a Hermitian positive-definite form on $\mathfrak{g}^c$ given by the formula
$B_{\tau}(X,Y)=-B(X,\tau(Y)).$
Choose a set of positive restricted roots $\Sigma^{+}\subset \Sigma.$  There is a system of positive roots $\Delta^{+}\subset \Delta$ such that
$$\Sigma^{+}=\{  \alpha|_{\mathfrak{a}} \ | \ \alpha\in \Delta^{+} \} - \{ 0 \} .$$
The subalgebra $\mathfrak{n}:=\sum_{\alpha\in \Sigma^{+}}\mathfrak{g}_{\alpha}$ is a nilpotent subalgebra of $\mathfrak{g}.$ The decomposition
$$\mathfrak{g}=\mathfrak{k}+\mathfrak{a}+\mathfrak{n}$$
is  the Iwasawa decomposition of $\mathfrak{g} .$ Notice that $\mathfrak{a}+ \mathfrak{n} $ is a solvable subalgebra and $\mathfrak{a}$ normalizes $\mathfrak{n} .$ On the Lie group level we have
$G=KAN,$
where $K$ is a maximal compact subgroup of $G$ and $N$ is a maximal unipotent subgroup of $G.$ Analogously we obtain the Iwasawa decompositions for $\mathfrak{h}$ and $\mathfrak{l}$
$$\mathfrak{h}=\mathfrak{k}_{h}+\mathfrak{a}_{h}+\mathfrak{n}_{h}, \ \ H=K_{h}A_{h}N_{h},$$
$$\mathfrak{l}=\mathfrak{k}_{l}+\mathfrak{a}_{l}+\mathfrak{n}_{l}, \ \ L=K_{l}A_{l}N_{l}.$$ 
In the same way we get a nilpotent Lie subalgebra of $\mathfrak{g}$ given by the negative restricted roots
$$\mathfrak{n}^{-}= \sum_{\alpha\in\Sigma^{-}}\mathfrak{g}_{\alpha},$$
as well as $\mathfrak{n}_{h}^{-},\mathfrak{n}_{l}^{-}$ for $\mathfrak{h}$ and $\mathfrak{l} .$  Since  $\theta(\mathfrak{g}_{\alpha})=\mathfrak{g}_{-\alpha}$ for all $\alpha\in\Sigma$ (\cite{knapp}, Proposition 5.9) we get $\theta (\mathfrak{n})=\mathfrak{n}^{-}$.
We will denote by $\Sigma_h$, $\Sigma_l$ the real root systems of $\mathfrak{h}$ and $\mathfrak{l}$, respectively, and consider $\mathfrak{n}_h$ and $\mathfrak{n}_l$ corresponding to $\Sigma_h$ and $\Sigma_l$.
\subsection{Properties of Lie algebra embeddings yielding compact Clifford-Klein forms}\label{subs21}
\begin{lemma}[\cite{botr}]\label{lemma:nh+nl}
Let $(\mathfrak{g},\mathfrak{h},\mathfrak{l})$ be a standard triple. Then we may assume that
$$\mathfrak{n}=\mathfrak{n}_{h}\oplus \mathfrak{n}_{l} , \ \ \mathfrak{n}^{-}=\mathfrak{n}_{h}^{-}\oplus \mathfrak{n}_{l}^{-}, \ \ \mathfrak{a}=\mathfrak{a}_{h}\oplus \mathfrak{a}_{l},$$
$$\theta (\mathfrak{h})=\mathfrak{h}, \ \ \theta (\mathfrak{l})=\mathfrak{l} .$$
\label{dobro}
\end{lemma}
We will denote by $\Sigma_h$, $\Sigma_l$ the restricted root systems of $\mathfrak{h}$ and $\mathfrak{l}$, respectively, corresponding to the above decompositions.
\begin{definition}
{\rm Embeddings $\mathfrak{h}\hookrightarrow\mathfrak{g}, \mathfrak{l} \hookrightarrow \mathfrak{g}$ fulfilling the conditions of Lemma \ref{dobro} will be called  good, and the corresponding subalgebras will be called well embedded.}
\end{definition}
\noindent In this terminology we may formulate the problem of classifying standard compact Clifford-Klein forms as a problem of describing Lie algebra triples given by good embeddings. We will say that  triples of Lie algebras $(\mathfrak{g},\mathfrak{h},\mathfrak{l})$ and $(\mathfrak{g},\mathfrak{h}',\mathfrak{l}')$ are isomorphic, if there exist inner automorphisms $\varphi_1,\varphi_2\in \operatorname{Int}(\mathfrak{g})$ such that $\mathfrak{h}'=\varphi_1(\mathfrak{h})$ and $\mathfrak{l}'=\varphi_2(\mathfrak{l})$.
\begin{proposition}[\cite{botr}]\label{prop:good}
\noindent Let $(\mathfrak{g},\mathfrak{h},\mathfrak{l})$ be a triple of Lie algebras which determines a standard compact semisimple Clifford-Klein form. Then there is an isomorphic  Lie algebra triple $(\mathfrak{g},\mathfrak{h}',\mathfrak{l}')$ such that $\mathfrak{h}'$ and $\mathfrak{l}'$ are well embedded.  
\end{proposition}

Assume that $0\not=X\in\mathfrak{t}$. Define
$$\Delta_m=\{\alpha\in\Delta\,|\,\alpha|_{\mathfrak{a}}=0\},$$
$$\Delta_m^{+}=\Delta^+\setminus\Delta_m,\Delta_m^{-}=\Delta^{-}\setminus\Delta_m,$$
$$\Delta_p^{+}=\{\alpha\in\Delta_m^{+}\,|\,\alpha (iX)>0\},\,\Delta_n^{+}=\{\alpha\in\Delta_m^{+}\,|\,\alpha(iX)<0\},$$
$$\Delta_p^{-}=\{\alpha\in\Delta_m^{-}\,|\,\alpha(iX)>0\},\,\Delta_n^{-}=\{\alpha\in\Delta_m^{-}\,|\,\alpha(iX)<0\},$$
$$\Delta_0^{+}=\{\alpha\in\Delta_m^{+}\,|\,\alpha(iX)=0\}, \ \Delta_0^{-}=\{\alpha\in\Delta_m^{-}\,|\,\alpha(iX)=0\} ,$$
$$\Delta_{0}=\Delta_{0}^{+}\cup \Delta_{0}^{-}.$$
Let us now formulate the basic technical tool used in the proof of the main theorem. This tool is a restriction on the root system of $\mathfrak{g}^c$ derived under the assumption that the triple $(\mathfrak{g},\mathfrak{h},\mathfrak{l})$ yields a standard compact Clifford-Klein form, but $\mathfrak{g}\not=\mathfrak{h}+\mathfrak{l}$.
\begin{theorem}[\cite{botr}]\label{thm:root-decomp} 
Assume that $(\mathfrak{g},\mathfrak{h},\mathfrak{l})$ is a standard triple of well embedded subalgebras and that $\mathfrak{g}\not=\mathfrak{h}+\mathfrak{l}$. We may assume that there exists a non-zero $X\in\mathfrak{t}$ such that $X$ is $B$-orthogonal to $\mathfrak{h}+\mathfrak{l}$.

Let
$$Z=\sum_{\alpha\in\Delta_p^{+}\cup\Delta_n^{+}}\mathfrak{g}_{\alpha}\subset\mathfrak{n}^c.$$
Let $\pi:\mathfrak{n}^c=\sum_{\alpha\in\Delta_m^+}\mathfrak{g}_{\alpha}\rightarrow Z$ be the natural projection, $Z_h=\pi(\mathfrak{n}_h^c), Z_l=\pi(\mathfrak{n}_l^c)$. Assume that $(\mathfrak{g},\mathfrak{h},\mathfrak{l})$ is a standard triple and that there exists a non-zero $X\in\mathfrak{t}$ such that $X\not\in (\mathfrak{h}+\mathfrak{l})$. We may assume that $X\in\mathfrak{t}$ and that $X$ is $B$-orthogonal to $\mathfrak{h}+\mathfrak{l} .$ Then:
\begin{enumerate}
\item There exists a basis of $Z_h$ of the form
$$S_h^i=x_{\alpha_i}+\sum_{l=1}^ka^i_{k+l}x_{\alpha_{k+l}},a^i_{k+l}\in\mathbb{C},\alpha_i\in\Delta_p^{+},\alpha_{k+l}\in\Delta_n^{+} .$$
\item For any $S_h^i\in Z_h$ $\alpha_i|_{\mathfrak{a}_h}\in\Sigma_h$.
\item Each complexified real root space $\mathfrak{h}_{\gamma}^c$ is spanned by vectors of the form
$$S_h^{i_1}+Q_1,...,S_h^{i_s}+Q_s,Q_{s+1},...,Q_{s+w},s+w=\dim\,{h}_{\gamma}^c,$$
where $\alpha_{i_1},...,\alpha_{i_s}$ are all roots from $\Delta_p^{+}$ whose restrictions onto $\mathfrak{a}_h$ coincide with $\gamma$, while all $Q_j$ satisfy the conditions
$$Q_j\in\sum_{\alpha\in\Delta_{0}, \ \alpha|_{\mathfrak{a}_{h}}=\gamma}\mathfrak{g}_{\alpha_{c}}.$$
\end{enumerate}
Analogous conditions hold for $\mathfrak{l}_{\gamma}^c$.
\end{theorem}
\begin{corollary}[\cite{botr}]\label{cor:split}
If $(\mathfrak{g}, \mathfrak{h}, \mathfrak{l}) ,$ with $\mathfrak{g}$ absolutely simple and split, yields a  standard compact Clifford-Klein form then $\mathfrak{g}= \mathfrak{h}+ \mathfrak{l} .$ Thus $(\mathfrak{g}, \mathfrak{h}, \mathfrak{l})$ is contained in Table \ref{tttab}.
\end{corollary}
Theorem \ref{thm:root-decomp} is quite technical, so we think it will be instructive to support it by example which somehow explains the consequences of the relations between complex and real root systems caused by the assumption $\mathfrak{g}\not=\mathfrak{h}+\mathfrak{l}$. These relations together with Lemma \ref{lemz} yield a contradiction.
\begin{example} {\rm We follow notation from Proposition \ref{prop:inequality}, the reader should consult it first. Let $\mathfrak{g}=\mathfrak{so}(n,n+2)$. Looking at Table 4 in \cite{ov} one can look at the Satake diagram in this case}
$$\dynkin[labels={\alpha_1,\alpha_2,\alpha_{n-2},\alpha_{n-1}, \alpha_{n},\alpha_{n+1}}, edge length=2cm, involutions={[in=120,out=60,relative]56}]{D}{oo.oooo}$$
{\rm and notice that $\dim\mathfrak{m}_0=1$ and that $\Delta_m=\emptyset$. It follows easily that necessarily 
$$X=H_{n+1},\Delta_p^+=\{\bar\alpha_n=\alpha_n+\alpha_{n+1}\},\Delta_p^{-}=\{\bar\alpha_n=\alpha_n-\alpha_{n+1}\}.$$
By Lemma \ref{lemz} every root in $\Delta_0-\Delta_m$ is a sum of a root from $\Delta_p^{\pm}\cup\Delta_n^{\pm}$ and a root from $\Delta_p^{\pm}\cup\Delta_n^{\pm}\cup \Delta_{m}$. Since $\Delta_{m}=\emptyset$ we get $\Delta=D_{n+1}$ and therefore it must be
$$\Delta=\{\pm\alpha_n\pm\alpha_{n+1}\},$$
but then the only possibility is $\Delta=D_2$, which implies $\mathfrak{g}=\mathfrak{so}(1,3)$ which never yields compact Clifford-Klein forms (with $\mathfrak{h}$ proper and of non-compact type) because the real rank of $\mathfrak{g}$ is 1.}
\end{example}

  We consider triples $(G,H,L)$ which yield Clifford-Klein forms up to an equivalence which is given by the following well-known  fact:
 if a triple $(G,H,L)$ yields a standard compact Clifford-Klein form then so does $(G,g_{1}Hg_{1}^{-1},g_{2}Lg_{2}^{-1})$ for any $g_{1},g_{2}\in G.$
This enables us to choose embeddings $\mathfrak{h}\hookrightarrow\mathfrak{g}$, $\mathfrak{l}\hookrightarrow\mathfrak{g}$ up to conjugacy in $G$, and we will do this throughout without further notice. By the general theory of compact Clifford-Klein forms and \cite{on}, Theorem 3.1 we can work on the Lie algebra level and study Lie algebra decompositions and their relation to the problem of standard compact Clifford-Klein forms.

Also, using the well known result \cite{bela} we restrict ourselves to the case when $\mathfrak{h}$ and $\mathfrak{l}$ are semisimple and are direct sums of ideals of non-compact type  (one can also consult \cite{botr}).

Put 
$$d(G):=\dim\,G/K,\,d(H):=\dim\,H/K_{h},\,  d(L):=\dim\,L/K_{l}.$$
Clearly, we can write $d(\mathfrak{g}), d(\mathfrak{h}),d(\mathfrak{l}).$
\begin{theorem}[\cite{kob}, Theorem 4.7]
Assume that $L$ acts properly on $G/H.$ The triple $(G,H,L)$ induces a standard compact Clifford-Klein form if and only if 
\begin{equation}
d(G)=d(H)+d(L),\,\text{or, equivalently}\,\,d(\mathfrak{g})=d(\mathfrak{h})+d(\mathfrak{l}).
\label{eq2}
\end{equation}
\label{koba}
\end{theorem}

Now we cite a theorem of T. Kobayashi which plays an important role in the proof of the main theorem of this article: it  eliminates some (otherwise possible)  triples $(\mathfrak{g},\mathfrak{h},\mathfrak{l})$ from the list of candidates.
\begin{theorem}[\cite{kob-d}, Theorem 1.5]\label{thm:kob-d} Let $G/H$ be a homogeneous space of reductive type. If $G/H$ has a compact Clifford-Klein form, then there does not exist a closed subgroup $G'$ reductive in $G$ satisfying the conditions
\begin{itemize}
\item $\mathfrak{a}_{g'}\subseteq\mathfrak{a}_h$,
\item $d(G')>d(H)$.
\end{itemize}
\end{theorem} 

In this work we consider embeddings of real Lie algebras $\mathfrak{h}\subset\mathfrak{g}$ and their complexifications $\mathfrak{h}^c\subset\mathfrak{g}^c$. Also, if for an  embedding of complex Lie algebras $\mathfrak{h}^c\subset\mathfrak{g}^c$ there exist real forms $\mathfrak{h}$ of $\mathfrak{h}^c$ and $\mathfrak{g}\subset\mathfrak{g}^c$ such that $\mathfrak{h}\subset\mathfrak{g}$, we say that the embedding $\mathfrak{h}\subset\mathfrak{g}$ is the real form of the complex embedding $\mathfrak{h}^c\subset\mathfrak{g}^c$.

Given a root $\alpha$ and a root space $\mathfrak{g}_{\alpha}$ we denote by $t_{\alpha}\in \mathfrak{j}^{c}$ the  vector determined by the equality $B(t_{\alpha},h)=\alpha(h)$ for all $h\in\mathfrak{j}^c$. Set $h_{\alpha}=2t_{\alpha}/B(t_{\alpha},t_{\alpha}).$
\begin{definition}  {\rm A Chevalley basis of $\mathfrak{g}^c$ with respect to $\mathfrak{j}^c$ is a basis of $\mathfrak{g}^c$  consisting of $x_{\alpha}\in\mathfrak{g}_{\alpha}$ and $h_{\alpha}$ with the following properties:}
\begin{enumerate}
\item $[x_{\alpha}, x_{-\alpha}]=-h_{\alpha}, \,\forall \alpha\in\Delta$,
\item {\rm for each pair $\alpha,\beta\in\Delta$ such that $\alpha+\beta\in\Delta$ the constants $c_{\alpha,\beta}\in\mathbb{C}$ determined by $[x_{\alpha},x_{\beta}]=c_{\alpha,\beta}x_{\alpha+\beta}$ satisfy $c_{\alpha,\beta}=c_{-\alpha,-\beta}$.}
\end{enumerate}
\end{definition}
\noindent Throughout, we work with a special choice of the Chevalley basis compatible with a non-compact real form $\mathfrak{g}$ of $\mathfrak{g}^c$.

\begin{lemma}[\cite{botr}] For the chosen $\mathfrak{j}^c$ in $\mathfrak{g}^c$ there is a base of $\sum_{\alpha\in\Delta}\mathfrak{g}_{\alpha}$ of the form
 $$\{ \tilde{x}_{\alpha} \ | \ \tilde{x}_{\alpha}\in \mathfrak{g}_{\alpha}, \ \alpha\in \Delta  \}, $$ 
such that for every $\alpha\in\Delta$
	      $$\tau (\tilde{x}_{\alpha}) = \tilde{x}_{-\alpha}, \  \ B(\tilde{x}_{\alpha}, \tilde{x}_{-\alpha})=-1.$$
\label{lemtau}
\end{lemma}

\section{Some properties of  root systems}
In this subsection  we present  Lemma \ref{lemma:root-sum} (proved in \cite{botr}) which describes subsets of the root system consisting of roots simultaneously vanishing on a pair of vectors.  Lemma \ref{lemma:root-sum} is used in the proof of Lemma \ref{lemz} which will be essential in the proof of the main result. Let $\Delta\subset \mathbb{R}^{n}$ be an indecomposable root system in the Euclidean space $(\mathbb{R}^{n},(,))$  (\cite{ov}, Chapter 3). The system $\Delta$ is not assumed to be reduced. Let $\Delta^{+}\subset \Delta$ be a subset of positive roots and  $\Pi\subset \Delta^{+}$  the subset of simple roots. There is a unique maximal root $\beta \in \Delta^{+}$ such that for every $\alpha\in \Delta^{+}$  vector $\beta - \alpha$ is a combination of simple roots with non-negative coefficients.

\begin{lemma}[\cite{botr}]\label{lemma:root-sum}
Let $\Delta\subset\mathbb{R}^n$ be an indecomposable root system and  $X,H\in\mathbb{R}^{n}- \{ 0 \}$ be  non-zero vectors. Define 
$$C_{X}:= \{ \alpha\in\Delta \ | \ \alpha (X) = 0   \} , \ \ C_{H}:= \{ \alpha\in\Delta \ | \ \alpha (H) = 0   \} .$$
 Then $\Delta \neq C_{X} \cup C_{H}.$ 
\end{lemma}

The following Lemma was proved in \cite{botr}, but it is so essential, that we repeat the proof here.
\begin{lemma}
None of the non-zero root vectors $x_{\beta}$ for  $\beta\in \Delta_{0}-\Delta_{m}$   centralizes $Z+\tau (Z) .$ 
\label{lemz}
\end{lemma}
\begin{proof}
Assume there exists $x_{\beta},\beta\in\Delta_0$ such that  $[x_{\beta},Z+\tau(Z)]=0$.  Take $H_{\beta}\in i\mathfrak{t}+\mathfrak{a}$, defined, as usual by the condition $B(\tilde{H},H_{\beta})=\beta(\tilde{H})$ for any $\tilde{H}\in i\mathfrak{t}+\mathfrak{a}$.  Denote by $\mathfrak{g}^{x_{\beta}}$ the centralizer of $x_{\beta}$ in $\mathfrak{g}^c$. By \cite{col}, Lemma 3.4.3
$$\mathfrak{g}^{x_{\beta}} = \oplus_{j\geq 0} (\mathfrak{g}^c\cap \mathfrak{g}_{j}),$$
where $\mathfrak{g}_j$ denote the eigenspaces of $\operatorname{ad}\,H_{\beta}$ (which are integers $j\geq 0$). Write $H_{\beta}=iT+H$, where $T\in\mathfrak{t},H\in\mathfrak{a}$. Note that $H\not=0$, since $\beta\in\Delta_0$. We have
$$H=\frac{1}{2}(H_{\beta}+\sigma(H_{\beta})).$$
Define
 $$ C_X=\{\alpha\in\Delta\,|\alpha(X)=0\},\,C_H=\{\alpha\in\Delta\,\,|\,\alpha(H)=0\}.$$
Consider the decomposition
$$\Delta=\Delta_{m}\cup\Delta_0\cup\Delta_p^{\pm}\cup\Delta_n^{\pm}.$$
 Note that for any $\mu\in\Delta_{m}$, $\mu(H)=0$. Assume that 
\begin{equation}
\Delta_p^{\pm}\cup\Delta_n^{\pm} \subset C_H. \label{eq5}
\end{equation}
Since $\Delta_0\subset C_X$, and all roots from $\Delta_{m}$ vanish on $H,$ equality (\ref{eq5}) would imply  
$$\Delta=C_X\cup C_H,$$
but this contradicts Lemma \ref{lemma:root-sum}. Thus there exists $\alpha,-\alpha\in\Delta_p^{\pm}\cup\Delta_n^{\pm}$ such that $\alpha(H)\not=0$. We have one of the following possibilities
\begin{enumerate}
	\item $\alpha(iT)\neq -\alpha(H).$ In this case (since one can choose between $\alpha$ and $-\alpha$) we may assume that $\alpha(H_{\beta})<0.$
	\item $\alpha(iT)= -\alpha(H).$ In this case we see that $\sigma (\alpha)(H_{\beta})=\alpha(-iT+H)=2\alpha(H)\neq 0$ and by Lemma \ref{sigma}, $\sigma (\alpha)\in \Delta_p^{\pm}\cup\Delta_n^{\pm}.$ In this case (since one can choose between $\sigma(\alpha)$ and $-\sigma(\alpha)$) we may assume that $\sigma(\alpha)(H_{\beta})<0.$
\end{enumerate}
Therefore in any case there exists $\alpha\in\Delta_p^{\pm}\cup\Delta_n^{\pm}$ such that $\alpha(H_{\beta})<0.$ But for any root vector $x_{\alpha}\in Z+\tau(Z)$ we have $\operatorname{ad}_{H_{\beta}}(x_{\alpha})=\alpha(H_{\beta})x_{\alpha}$ and so for $x_{\alpha}\in\mathfrak{g}^{x_{\beta}},$ $j=\alpha(H_{\beta})$. This is a contradiction with the inequality $j\geq 0$. The proof of the Lemma is complete. Note that Lemma \ref{lemtau} is tacitly used here.
\end{proof}

\section{Proof of Theorem \ref{thm:class}: eliminating the regular case}\label{sec:regular}

\subsection{Assumptions}
In this section we consider the following triples $(\mathfrak{g},\mathfrak{h},\mathfrak{l})$:
\begin{enumerate}
	\item $\mathfrak{g}$ is an absolutely simple Lie algebra,
	\item $\mathfrak{h}, \mathfrak{l}$ are absolutely simple (of non-compact type) subalgebras of $\mathfrak{g}$ and the triple $(\mathfrak{g}, \mathfrak{h}, \mathfrak{l})$ gives a standard compact Clifford-Klein form,
	\item the triple is not contained in Table \ref{tttab} and thus there exists $0\neq X\in \mathfrak{m}_{0}$ orthogonal to $\mathfrak{h}+ \mathfrak{l}$,
	\item $\mathfrak{h}^{c}$ is regular in $\mathfrak{g}^{c} ,$ that is, some Cartan subalgebra of $\mathfrak{g}^{c}$ normalizes $\mathfrak{h}^{c}.$
\end{enumerate}
In this section we prove the "regular" part of Theorem \ref{thm:class}, that is, the following.
\begin{theorem}\label{thm:centr}
Assume that neither $G/H$ nor $H$ is compact. Under the assumptions of this section   the pair $(\mathfrak{g}, \mathfrak{h})$ cannot admit compact standard Clifford-Klein forms.
\end{theorem}
\subsection{Plan of the proof}
The plan of the proof  is as follows.
\begin{enumerate}
\item The assumptions that $(\mathfrak{g},\mathfrak{h},\mathfrak{l})$ induces a compact standard Clifford-Klein form and that there exists $0\neq X\not\in \mathfrak{h}+\mathfrak{l}$ yield the conclusions of Theorem \ref{thm:root-decomp}, that is, some particular properties of the spaces $Z_h$ and $Z_l$ with  bases $S_h^i$ and $S_l^i$ (statement 1),  the relation between $S_h^i$ and $\Sigma_h$ ( statement 2), and the form of the base of the complexified real root spaces $\mathfrak{h}_{\gamma}^c$ (statement 3).
\item The assumption of the regularity of $\mathfrak{h}^c$ together with the assumption $X\not\in\mathfrak{h}+\mathfrak{l}$ yields the existence of a Cartan subalgebra $\mathfrak{c}\subset\mathfrak{g}^c$  such that $\mathfrak{a}_h\subset\mathfrak{c}$ and an element $0\not=Y\in\mathfrak{c}$ such that $Y\not\in\mathfrak{m}_0$, but $Y\in\mathfrak{z}_{\mathfrak{g}}(\mathfrak{h})$ and $[Y,iX]=0$.
\item Statements 1) and 2) are contradictory, and, therefore, $\mathfrak{g}=\mathfrak{h}+\mathfrak{l}$, as required.
\end{enumerate}

\subsection{Deriving the contradiction between 1) and 2)}
Let us show first, how to derive the required contradiction. Put
$$Z_p=\sum_{\alpha\in\Delta_p^{\pm}}\mathfrak{g}_{\alpha}, Z_n=\sum_{\alpha\in\Delta_n^{\pm}}\mathfrak{g}_{\alpha}, O=\sum_{\alpha\Delta_0}\mathfrak{g}_{\alpha}.$$
By 2), 
\begin{equation}\label{eq:bracket}
[Y,iX]=0, [Y,\mathfrak{a}_h]=0.
\end{equation}
 This implies
$$[Y,Z_n]\subset Z_n, [Y, Z_p]\subset Z_p, [Y,O]\not\subset (Z_p+Z_n)\setminus \{ 0 \} .$$
Indeed, $\mathfrak{g}_{\alpha} \subset Z_p$ if and only if $\alpha(iX)>0$ and $\alpha(A^+)>0$ for any $A^+$ from the interior of the positive Weyl chamber in $\mathfrak{a}_h$ (we consider Weyl chambers in $\mathfrak{a}_h$). Let $U\in\mathfrak{g}_{\alpha}\subset Z_p$. The Jacobi identity and (\ref{eq:bracket}) imply 
$$[iX,[Y,U]]=\alpha(iX)[Y,U],\,\text{and}\,\,[iX,[Y,A^+]]=\alpha(A^+)[Y,U].$$
Thus, $[Y,U]$ is contained in a sum of root spaces determined by roots which are positive on $iX$ and on $A^+$, as required. Similar arguments work for $Z_n$ and $O$. 

Since $[Y,\mathfrak{h}^c]=0$, by Theorem \ref{thm:root-decomp} (statements 2 and   3), $[Y,S_h^i+Q_i]=0$ for all $i=1,...,k$. Theorem \ref{thm:root-decomp} (statement 1) implies 
$$[Y,x_{\alpha_i}]=0,\,\,\text{for all}\,\,\alpha_i\in\Delta_p^{\pm}.$$
 It follows that $[Y,\sigma(x_{\alpha_i})]=0$ as well, and, therefore, $[Y,x_{\alpha_j}]=0$ for any $\alpha_j\in\Delta^{\pm}_n$. Lemma \ref{lemz} yields $[Y,\mathfrak{a}]=0$ which contradicts  2), that is, the assumption $Y\not\in\mathfrak{m}_0$.

\subsection{Constructing $Y\in\mathfrak{c},Y\not\in\mathfrak{m}_0,Y\in\mathfrak{z}_{\mathfrak{g}}(\mathfrak{h}),[Y,iX]=0$}\label{subsect:y}
We will construct the required $Y$ in the centralizer $\mathfrak{z}_{\mathfrak{g}^c}(\mathfrak{a}_h)$.
\subsubsection{The structure of $\mathfrak{z}_{\mathfrak{g}^c}(\mathfrak{a}_{h})$}
One easily sees that
$$\mathfrak{z}_{\mathfrak{g}^c}(\mathfrak{a}_h)=\mathfrak{d}^c=\mathfrak{m}_0+(\mathfrak{a}_h^{\perp})^c+\mathfrak{a}_{h}+\sum_{\alpha\in\Delta_0,\alpha|_{\mathfrak{a}_{h}}=0}\mathfrak{g}_{\alpha},$$
where $(\mathfrak{a}_h^{\perp})^c$ is the complexification of the orthogonal complement to $\mathfrak{a}_{h}$ in $\mathfrak{a}$.
This is a reductive subalgebra in $\mathfrak{g}^c$. Since $\mathfrak{h}^c$ is regular, there exists a Cartan subalgebra $\mathfrak{c}\subset\mathfrak{g}^c$ such that $[\mathfrak{c},\mathfrak{h}^c]\subset\mathfrak{h}^c.$ Clearly, $\mathfrak{h}^c\cap\mathfrak{c}$ is a Cartan subalgebra in $\mathfrak{h}^c$. Since any two Cartan subalgebras in $\mathfrak{h}^c$ are conjugate by elements $\operatorname{Ad}(h)$ (where $h$ is an element in the connected subgroup $H^c\subset G^c$ corresponding to $\mathfrak{h}^c$), any Cartan subalgebra of $\mathfrak{h}^c$ is of the form $\operatorname{Ad}(h)(\mathfrak{c})\cap\mathfrak{h}^c$. It easily follows that without loss of generality one may assume that $\mathfrak{a}_h^c\subset\mathfrak{c}$. Since $\mathfrak{d}^c$ centralizes $\mathfrak{a}_h$ one finally obtains 
\begin{equation}
\mathfrak{a}_h^c\subset\mathfrak{c}\subset\mathfrak{d}^c.
\end{equation}
\begin{lemma}\label{lemma:invar} The Cartan subalgebra $\mathfrak{c}$ can be chosen in a way to satisfy $(4)$ and to be $\sigma$- and $\theta$-invariant.
\end{lemma}
\begin{proof}
Clearly,  $\sigma (\mathfrak{a}_{h}) = \mathfrak{a}_{h}$.  It follows that $\sigma (\mathfrak{d}^{c}) = \mathfrak{d}^{c} .$ Denote by $\mathfrak{d}^{\sigma}$ the real form of $\mathfrak{d}^{c}$ given by $\sigma .$ We see that $\mathfrak{d}^{\sigma}$ is the centralizer of $\mathfrak{a}_{h}$ in $\mathfrak{g} .$ Put $\mathfrak{c}_{h}:=\mathfrak{c}\cap \mathfrak{h}^{c}.$ After conjugating $\mathfrak{c}$ by an element of the centralizer of $\mathfrak{a}_{h}$ in $H^{c}$ we get $\sigma (\mathfrak{c}_{h}) = \mathfrak{c}_{h} .$ Therefore $\mathfrak{c}_{h}^{\sigma}$ (the real form of $\mathfrak{c}_{h}$ with respect to $\sigma$) is in $\mathfrak{d}^{\sigma} .$ Take a $\sigma$-stable Cartan subalgebra $\mathfrak{c}'$ of the normalizer of $\mathfrak{h}^{c}$ in $\mathfrak{d}^{c} $ containing $\mathfrak{c}_{h}^{\sigma} .$ One checks that $[\mathfrak{c}' , \mathfrak{h}^{c}]\subset \mathfrak{h}^{c} $ and $\sigma (\mathfrak{c}')\subset \mathfrak{c}',$ as desired. By a similar argument (conjugating $(\mathfrak{c}')^{\sigma}$ inside $G$) we get $\theta (\mathfrak{c}') = \mathfrak{c}'.$
\end{proof}

\begin{lemma}
The following equality is satisfied:
 $$\mathfrak{a}_{h}^{\perp} \cap \mathfrak{c}^{\sigma} =\{ 0 \}. $$ 

\label{lemaorto}
\end{lemma}
\begin{proof}
Let $Y\in \mathfrak{a}_{h}^{\perp}\cap \mathfrak{c}^{\sigma} .$ For every $1\leq i\leq k$ we get the following: 
$$[Y,S_{h}^{i}+Q_{i}+\sigma (S_{h}^{i}+Q_{i})]=0 ,$$
$$[Y,i(S_{h}^{i}+Q_{i}-\sigma (S_{h}^{i}+Q_{i}))]=0,$$ 
 This easily follows from the fact that for any $i=1,...,k$, vectors 
$$S_h^i+Q_i+\sigma(S_h^i+Q_i)\,\,  \text{and}\,\, i(S_{h}^{i}+Q_{i}-\sigma (S_{h}^{i}+Q_{i})$$
 belong to $\mathfrak{h}$, and  the properties of the Killing form. Indeed, one can notice that $[Y,h]\perp \mathfrak{h}$ for any $h\in\mathfrak{h}$. On the other hand, $[Y,h]\in\mathfrak{h}$ since $Y$ normalizes $\mathfrak{h}$.  Now observe that  since $i$ takes all values from $1$ to $k$,  $Y$ normalizes all root vectors of $\Delta_{p}^{\pm}$ and $\Delta_{n}^{\pm} .$ By Lemma \ref{lemz} every root of $\Delta_{0}^{\pm}$ is of the form $\alpha_{1}+\alpha_{2}$ for some $\alpha_{1}\in \Delta_{p}^{\pm} \cup \Delta_{n}^{\pm} $ and $\alpha_{2}\in \Delta_{p}^{\pm} \cup \Delta_{n}^{\pm}\cup \Delta_{m} .$  This implies that $Y$ centralizes $\mathfrak{n}$ and so $Y=0.$
\end{proof}

  Consider the semisimple part of the reductive subalgebra $\mathfrak{d}^c$, that is 
$$\mathfrak{d}_{s}^{c}:=[\mathfrak{d}^{c},\mathfrak{d}^{c}].$$
Define 
$$\Phi=\{\alpha\in\Delta_0\,|\,\alpha|_{\mathfrak{a}_h}=0\}.$$
\begin{lemma}\label{lemma:Phi}
At least one of the following is valid:
\begin{enumerate}
	\item  $\mathfrak{d}^{c}_{s}$ is not simple,

	\item $iX$ is in the center of $\mathfrak{d}^{c} .$
\end{enumerate}
\end{lemma}
\begin{proof}
 If $iX$ is in the center of $\mathfrak{d}^{c}$ then the proof is complete. So assume that  $\alpha (iX) \neq 0$ for some $\alpha\in\Delta_{m} .$ By Lemma \ref{lemaorto}, $\Phi\not=\emptyset$.  If $\beta\in \Phi$ then $\alpha \pm \beta$ is not a root since otherwise $\alpha \pm\beta \in\Delta_{p}^{\pm}\cup \Delta_{n}^{\pm}$ and $(\alpha \pm\beta)|_{\mathfrak{a}_{h}} =0,$ which is impossible (by Theorem \ref{thm:root-decomp}). Take any $\alpha_{1}\in \Phi .$ It follows that for any $\gamma\in\Delta_{m},$ $\gamma (iX)\neq 0$ we have $\gamma\perp\alpha_{1} .$ But now every root of $(\mathfrak{d}_{s}^{c}, (\mathfrak{t}^{c}\oplus\mathfrak{a}^{c})\cap \mathfrak{d}_{s}^{c})$ is orthogonal to $iX$ or to $\alpha_{1} $ and so all roots are contained in two hyperplanes (but not all roots are in the first one, nor in the second one). Lemma \ref{lemma:root-sum} implies that $\mathfrak{d}_{s}^{c}$ is not simple.
\end{proof}

\begin{corollary}\label{cor:decomp}
Lie algebra $\mathfrak{d}^{c}$ decomposes into a direct sum of Lie algebras
 $$\mathfrak{d}^{c}=\mathfrak{d}_{1}\oplus \mathfrak{d}_{2}\oplus \mathfrak{a}_{h}^{c}$$
 such that 
\begin{enumerate}
\item $iX\in \mathfrak{d}_{1}$ 
\item  $\mathfrak{a}_{h}^{\perp}\oplus \Sigma_{\alpha\in\Delta_{0}, \ \alpha |_{\mathfrak{a}_{h}}=0} \mathfrak{g}_{\alpha} \subset \mathfrak{d}_{2} .$
\item $\sigma(\mathfrak{d}_2)\subset\mathfrak{d}_2, \theta(\mathfrak{d}_2)\subset\mathfrak{d}_2$.
\end{enumerate}
\end{corollary}
\begin{proof}
Note that $\mathfrak{a}_{h}^{c}$ is in the center of $\mathfrak{d}^{c} .$ 
If $iX\in\mathfrak{z}(\mathfrak{d}^{c})$, then put $\mathfrak{d}_{1}:=\textrm{Span}_{\mathbb{C}}(iX)$ and take the ideal $\mathfrak{d}_{2}$  complementary to $\textrm{Span}_{\mathbb{C}}(iX)\oplus \mathfrak{a}_{h}^{c}.$ 
Otherwise,  assume that $\alpha_{1}\in\Phi \neq \emptyset $ and $\alpha (iX) \neq 0$ for some $\alpha\in\Delta_{m} .$ Define hyperplanes $s_{iX}$ and $s_{\alpha_1}$ in  $(\mathfrak{t}^{c}\oplus\mathfrak{a}^{c})\cap \mathfrak{d}_{s}^{c} ,$ orthogonal to $iX$ and $\alpha_{1},$   respectively. All roots of $(\mathfrak{d}_{s}^{c}, (\mathfrak{t}^{c}\oplus\mathfrak{a}^{c})\cap \mathfrak{d}_{s}^{c})$ lie in these two hyperplanes.   We see that if $\tilde{\mathfrak{d}}$ is a simple ideal of $\mathfrak{d}_{s}^{c}$, then (by Lemma \ref{lemma:root-sum}) all roots of $(\tilde{\mathfrak{d}},\tilde{\mathfrak{d}}\cap(\mathfrak{t}^{c}\oplus\mathfrak{a}^{c}))$ lie either in $s_{iX} ,$ or in $s_{\alpha_{1}} .$ Therefore we may take $\mathfrak{d}_{2}$ to be the sum of all simple ideals whose roots lie in $s_{iX}+(\mathfrak{a}_{h}^{\perp})^{c}.$ Then $\mathfrak{d}_{1}$ is the ideal of $\mathfrak{m}_{0}^{c}$ complementary to $\mathfrak{d}_{2}\cap \mathfrak{m}_{0}^{c} .$ 

Finally, the third part of the Corollary follows, since $\sigma(X)=X$ and $\theta(X)=X$.
\end{proof}

Let $\mathfrak{k}_d$  denote the subalgebra  of $\mathfrak{g}$ generated by vectors
$$ x_{\alpha} + \sigma (x_{\alpha}) +\theta (x_{\alpha} + \sigma (x_{\alpha})), \ i(x_{\alpha} - \sigma (x_{\alpha})) +\theta (i(x_{\alpha} - \sigma (x_{\alpha}))), \forall\alpha\in\Phi ,$$
and $\mathfrak{p}_d$ denote the subspace in $\mathfrak{g}$ generated by $\mathfrak{a}_h^{\perp}$ and all vectors of the form 
$$
 x_{\alpha} + \sigma (x_{\alpha}) -\theta (x_{\alpha} + \sigma (x_{\alpha})), \ i(x_{\alpha} - \sigma (x_{\alpha})) -\theta (i(x_{\alpha} - \sigma (x_{\alpha}))), \forall \alpha\in\Phi.
$$
From the very definition of $\mathfrak{k}_d$ and $\mathfrak{p}_d$, as well as the definition of $\mathfrak{d}^c$ and taking into consideration Corollary \ref{cor:decomp} one obtains the inclusions
\begin{equation}
\mathfrak{k}_d\oplus\mathfrak{p}_d\subset\mathfrak{d}_2^{\sigma}\subset\mathfrak{d}^{\sigma}\subset\mathfrak{k}_d\oplus\mathfrak{p}_d\oplus\mathfrak{m}_0\oplus\mathfrak{a}_h.
\end{equation} 

\begin{corollary}\label{cor:ah}
The following  is valid
$$\mathfrak{c}^{\sigma} \cap (\mathfrak{p}_{d}\oplus \mathfrak{a}_{h}) = \mathfrak{a}_{h } .$$
\end{corollary}
\begin{proof}
It is straightforward to see that $\mathfrak{a}_{h}\subset \mathfrak{c}^{\sigma} .$ Take any $Y\in \mathfrak{p}_{d}\cap \mathfrak{c}^{\sigma} .$  Let $K_d$ be a connected closed (reductive) subgroup of $K$ corresponding to $\mathfrak{k}_d$. Since $\mathfrak{a}_h^{\perp}$ is maximal abelian in $\mathfrak{p}_d$, there is an element $k\in K_d$ such that $\operatorname{Ad}_k(Y)\in\mathfrak{a}^{\perp}_h$.

 Since $\mathfrak{k}_{d}\subset \mathfrak{d}_{2}^{\sigma}$,  it centralizes $iX$ and $\mathfrak{a}_{h}$ and so for any $g\in K_{d}$  the subalgebra $\operatorname{Ad}_{g}\mathfrak{h}$ is well embedded (since $\operatorname{Ad}_{g}\mathfrak{n}_{h}^{\pm}\subset \mathfrak{n}^{\pm} ,$ $\operatorname{Ad}_{g}\mathfrak{a}_{h}=\mathfrak{a}_{h}$ and $\theta (\operatorname{Ad}_{g}\mathfrak{h}) = \operatorname{Ad}_{g}\mathfrak{h}$). Also $\mathfrak{a}_{h}^{\perp}$ is orthogonal to $\operatorname{Ad}_{g}\mathfrak{h} .$  Now, one can repeat the argument of the proof of Lemma \ref{lemaorto} (taking the appropriate Lie brackets) applying it to $\operatorname{Ad}_{k}Y$ and $\operatorname{Ad}_{k}\mathfrak{h} .$
\end{proof}
\subsubsection{Construction of $Y$}
Let $D_2^{\sigma}$ denote the connected closed subgroup of $G$ corresponding to $\mathfrak{d}_2^{\sigma}$. 
\begin{lemma}\label{lemma:d2-good} For any $g\in D_2^{\sigma}$ the embedding $\operatorname{Ad}_g(\mathfrak{h})\hookrightarrow\mathfrak{g}$ is good. Thus, Theorem \ref{thm:root-decomp} can be applied to this embedding.
\end{lemma}
\begin{proof}
 Since  $\mathfrak{d}_{2}^{\sigma}$ centralizes $iX$ and $\mathfrak{a}_{h}$ one derives that  for any $g$ in the maximal compact subgroup of $D_{2}^{\sigma}$ the subalgebra $Ad_{g}\mathfrak{h}$ is well embedded (because $\textrm{Ad}\mathfrak{n}_{h}^{\pm}\subset \mathfrak{n}^{\pm} ,$ $\textrm{Ad}\mathfrak{a}_{h}=\mathfrak{a}_{h}$ and $\theta (\textrm{Ad}_{g}\mathfrak{h}) = \textrm{Ad}_{g}\mathfrak{h}$).
\end{proof}
\begin{lemma}\label{lemma:y} For any $\alpha\in\Phi$ vector
 $$0\neq \tilde{Y}:= x_{\alpha}+\sigma (x_{\alpha})+\theta (x_{\alpha}+\sigma (x_{\alpha}))$$ 
does not belong to $\mathfrak{h}$, and for any $g\in D_2^{\sigma}$ vector $Y=\operatorname{Ad}_g(\tilde Y)$ does not belong to $\operatorname{Ad}_g(\mathfrak{h})$.
\end{lemma}
\begin{proof}
Since $x_{\alpha}+\sigma (x_{\alpha}) \in \mathfrak{n}$, there exist $H\in \mathfrak{n}_{h}$ and $L\in \mathfrak{n}_{l}$ such that $x_{\alpha}+\sigma (x_{\alpha}) = H+L$ and so $\tilde{Y}= H+ \theta (H) +L+ \theta (L).$ If $\tilde{Y}\in \mathfrak{h}$ then $L+\theta (L)\in \mathfrak{h}.$ But
$$\{ W+\theta (W) \ | \ W\in\mathfrak{n} \} = \{  W+\theta (W) \ | \ W\in\mathfrak{n}_{h} \} \oplus \{ W+\theta (W) \ | \ W\in\mathfrak{n}_{l}  \} $$
and thus $L+\theta (L) = H_{1} +\theta (H_{1})$ for some $H_{1}\in \mathfrak{n}_{h} .$ As $\mathfrak{n}\cap \theta (\mathfrak{n}) =\{ 0 \} $ therefore $L\in \mathfrak{n}_{h} $ and so $L=0.$ Thus $x_{\alpha} + \sigma (x_{\alpha}) \in \mathfrak{n}_{h}$ and $\alpha |_{\mathfrak{a}_{h}}=0,$ a contradiction. The second claim is clear. 
\end{proof}

Now we prove the main proposition which yields $Y$ as in the title of Subsection \ref{subsect:y}.
\begin{proposition}\label{prop:centralizer}
There exists a non-zero $Y\in \mathfrak{c}^{\sigma}\cap \mathfrak{d}_{2}^{\sigma},$ $Y\notin \mathfrak{m}_{0}$ which centralizes $\mathfrak{h}.$
\end{proposition}
\begin{proof} For the convenience of the reader we first summarize facts needed in the course of the proof.
\begin{enumerate}
\item For  any semisimple $\mathfrak{g}$ with a Cartan involution $\theta$ and any $\theta$-invariant Cartan subalgebra $\mathfrak{b}$ one has the decomposition 
$$\mathfrak{b}=\mathfrak{b}_{+}\oplus\mathfrak{b}_{-}.$$
By an abuse of language we will call $\mathfrak{b}_{+}$ the compact part of $\mathfrak{b}$ and $\mathfrak{b}_{-}$ non-compact part.
\item By Lemma \ref{lemma:invar} $\mathfrak{c}^{\sigma}$ is a $\theta$-invariant subalgebra of $\mathfrak{d}^{\sigma}$ and, hence,
$$\mathfrak{c}^{\sigma}=\mathfrak{c}^{\sigma}_{+}\oplus\mathfrak{c}^{\sigma}_{-}$$
and by Corollary \ref{cor:decomp}
$$\mathfrak{c}_{-}^{\sigma}=\mathfrak{a}_h,\mathfrak{c}^{\sigma}=\mathfrak{c}_{+}^{\sigma}\oplus\mathfrak{a}_h.$$
\item By Corollary \ref{cor:ah}
$$\mathfrak{c}^{\sigma}_{+}=\mathfrak{c}^{\sigma}\cap\mathfrak{d}_2^{\sigma}.$$
\item It is easy to see that any two compact parts of any two $\theta$-invariant Cartan subalgebras (with compact parts of the same dimension) in any semisimple real Lie algebra $\mathfrak{g}$ are conjugate, and, in particular, any two Cartan subalgebras in $\mathfrak{d}_2^{\sigma}$ are conjugate by an element $\operatorname{Ad}_g,g\in D_2^{\sigma}$. Therefore, any compact part of any Cartan subalgebra in $\mathfrak{d}_2^{\sigma}$ is conjugate to $\mathfrak{d}_2^{\sigma}\cap\mathfrak{c}^{\sigma}$.
\end{enumerate}
Now we complete the proof using 1)-4) above. Take
$$\tilde Y=x_{\alpha}+\sigma(x_{\alpha})+\theta(x_{\alpha}+\sigma(x_{\alpha}),\alpha\in\Phi.$$
By formula $(5)$, $\tilde Y\in\mathfrak{k}_d\subset\mathfrak{d}_2^{\sigma}$. Since $\mathfrak{k}_d$ is compact, $\tilde Y$ belongs to some compact part of some Cartan subalgebra in $\mathfrak{d}_2^{\sigma}$. It follows that $\tilde Y$ is conjugate to element $Y\in \mathfrak{c}^{\sigma}_{+}=\mathfrak{c}^{\sigma}\cap\mathfrak{d}_2^{\sigma}$. Thus, $Y=\operatorname{Ad}_g(\tilde Y)$ for some $g\in D_2^{\sigma}$. By Lemma \ref{lemma:d2-good} we can replace the pair $(\mathfrak{g},\mathfrak{h})$ by $(\mathfrak{g},\operatorname{Ad}_g(\mathfrak{h})$ and $\tilde Y$ by $Y\in\mathfrak{c}^{\sigma}\cap\mathfrak{d}_2^{\sigma}$. By construction, $Y\not\in\mathfrak{m}_0$ and by Lemma \ref{lemma:y} $Y\not\in\operatorname{Ad}_g(\mathfrak{h})$, hence one can also choose $Y\perp\operatorname{Ad}_g(\mathfrak{h})$. Since $\mathfrak{h}^c$ is regular, this is the same as $Y\in\mathfrak{z}_{\mathfrak{g}}(\operatorname{Ad}_g(\mathfrak{h}))$. Now, by some abuse of notation we write this as in the title.

\end{proof}

\section{Completion of proof of  Theorem \ref{thm:class}: non-regular case }\label{sect:ncomp-dim}

In view of Theorem \ref{thm:centr} we do not assume that the embeddings $\mathfrak{h}^c, \mathfrak{l}^c\hookrightarrow\mathfrak{g}^c$ are regular. Instead,  we assume that $\mathfrak{h}^c$ and $\mathfrak{l}^c$ are simple. 

The proof is based on a careful consideration of possibilities for the real forms of the embeddings $\mathfrak{h}^c\hookrightarrow\mathfrak{g}^c$, therefore, we postpone the calculations to the very end of the article. In this section we explain the main steps of the proof. 
\subsection{Terminology and notation} 
We begin by introducing some convenient (for us) terminology. Since $\mathfrak{g}^c$ is a classical Lie algebra, that is, $\mathfrak{g}^c$ can be $\mathfrak{sl}(V),\mathfrak{so}(V)$, or $\mathfrak{sp}(V)$,  we get linear representations $\mathfrak{h}^c\hookrightarrow\mathfrak{g}^c$ in the complex vector space $V$ and call them  {\it linear} if $\mathfrak{g}^c=\mathfrak{sl}(V)$, {\it orthogonal}, if $\mathfrak{g}^c=\mathfrak{so}(V)$ and {\it symplectic}, if $\mathfrak{g}^c=\mathfrak{sp}(V).$

We single out {\it tensor representations}, that is, the representations of the form
\begin{enumerate}
\item $\mathfrak{h}^c=\mathfrak{so}(V_1)\oplus\mathfrak{so}(V_2)\hookrightarrow\mathfrak{g}^c=\mathfrak{sl}(V_1\otimes V_2)$,
\item $\mathfrak{h}^c=\mathfrak{so}(V_1)\oplus\mathfrak{so}(V_2)\hookrightarrow\mathfrak{g^c}=\mathfrak{so}(V_1\otimes V_2$,
\item $\mathfrak{h}^c=\mathfrak{sp}(V_1)\oplus\mathfrak{sp}(V_2)\hookrightarrow\mathfrak{g}^c=\mathfrak{sp}(V_1\otimes V_2).$
\end{enumerate}
The third type of subalgebras we need has the form
\begin{enumerate}
\item $\mathfrak{h}^c=\mathfrak{so}(k,\mathbb{C})\oplus\mathfrak{so}(n-k,\mathbb{C})\hookrightarrow\mathfrak{g}^c=\mathfrak{so}(n,\mathbb{C})$,
\item $\mathfrak{h}^c=\mathfrak{sp}(k,\mathbb{C})\oplus\mathfrak{sp}(n-k,\mathbb{C})\hookrightarrow\mathfrak{g}^c=\mathfrak{sp}(n,\mathbb{C}).$
\end{enumerate}
We will call such subalgebras {\it symmetric} (they are fixed points of involutive automorphisms of $\mathfrak{g}^c$, because they are maximal in $\mathfrak{g}^c$ and  Dynkin's classification of such subalgebras is up to conjugation). Note that we do not consider all possible tensor representations or symmetric subalgebras, but only the listed ones. Alternatively we refer to Table \ref{tab-max}, where maximal subalgebras of maximal rank are presented. 

\begin{definition} {\rm We say that $\mathfrak{h}^c$ is an irreducible/reducible subalgebra in $\mathfrak{g}^c$ if the linear representation $\mathfrak{h}^c\hookrightarrow\mathfrak{g}^c$ is irreducible/reducible. }
\end{definition}
\subsection{Maximal subalgebras of classical complex simple Lie algebras}
In this subsection we summarize Dynkin's results on maximal subalgebras in classical Lie algebras in \cite{d} and \cite{ov} which we formulate as Theorem  \ref{thm:max-ov} and Theorem \ref{thm:d-irred}.
\begin{theorem}[\cite{ov}]\label{thm:max-ov} All maximal proper subalgebras in classical complex Lie algebras fall into the classes below, up to conjugation in $\mathfrak{g}^c$:
\vskip6pt
\centerline{\bf I: Reducible}
\vskip6pt
\begin{itemize}
\item maximal parabolic (I.1),
\item symmetric (I.2)
\end{itemize}
\centerline{\bf II: Irreducible}
\vskip6pt
\begin{itemize}
\item simple (II.1)
\item tensor representations (II.2)
\end{itemize}
\end{theorem}
Not all irreducible simple subalgebras are maximal, however, we have the following ( $\omega_r$ is $0$ if the representation is linear, $1$ if it is orthogonal, and $(-1)$ if symplectic).
\begin{theorem}[\cite{d}]\label{thm:d-irred} Irreducible simple subalgebras in classical Lie algebras are maximal, with exceptions given by the following two tables.
\begin{center}
 \begin{table}[h]
 \centering
 {\footnotesize
 \begin{tabular}{| c | c | c| }
   \hline
   \multicolumn{3}{|c|}{ \textbf{\textit{Irreducible triples}}} \\
   \hline                        
   $\text{Inclusion 1}$ & $\text{Inclusion 2}$ & $\dim$ of the representation   \\
   \hline
		$C_n\subset A_{2n-1}$ & $A_{2n-1}\subset\tilde{\mathfrak{g}}^c$ & $\binom{2n-1+k}{k}$ \\
		\hline
	           $B_n\subset A_{2n}$ & $A_{2n}\subset\tilde{\mathfrak{g}}^c$ & $ \binom{2n+1}{k}$  \\
		\hline
		$D_n\subset A_{2n-1}$ & $A_{2n-1}\subset\tilde{\mathfrak{g}}^c$ & $ \binom{2n}{k}$ \\
		\hline
		$A_n\subset A_{\frac{(n-1)(n+2)}{2}}$ & $A_{\frac{(n-1)(n+2)}{2}}\subset\tilde{\mathfrak{g}}^c$ & $ \frac{1}{8}(n-1)(n+1)(n+2)$ \\
		\hline
		$A_n\subset A_{\frac{n(n+3)}{2}}$ & $ A_{\frac{n(n+3)}{2}}\subset\tilde{\mathfrak{g}}^c$   & $\frac{1}{8}(n+1)(n+2)(n+3)$ \\
		\hline
		$B_n\subset D_{n+1}$ & $D_{n+1}\subset\tilde{\mathfrak{g}}^c$ & $ \prod_{s=1}^{n-1}\binom{k+2s-1}{k}/\binom{k+s}{k}$ \\
		\hline
 \end{tabular}
 }
 \caption{
 }
 \label{tttab3}
 \end{table}
\end{center}

\begin{center}
 \begin{table}[h]
 \centering
 {\footnotesize
 \begin{tabular}{| c | c | c| }
   \hline
   \multicolumn{3}{|c|}{ \textbf{\textit{Non-maximal irreducible subalgebras}}} \\
   \hline                        
   $\text{Type of inclusion}$ & $\text{dimension}$ & $\omega_r$   \\
   \hline
		$A_n\subset A_l$ & $3\binom{n+2}{4}$ & $0$ \\
		\hline
		$A_n\subset A_l$ & $3\binom{n+3}{4}$ & $0$ \\
		\hline
	           $B_{2n+1}\subset \mathfrak{g}^c$ & $\prod_{s=1}^{2n}{\binom{k+s-1}{k}/ \binom{k+s}{k}}$ & $ (-1)^{n+1)k}$  \\
		\hline

 \end{tabular}
 }
 \caption{
 }
 \label{tttttab2}
 \end{table}
\end{center}

\end{theorem}

\subsection{Steps of proof of Theorem \ref{thm:class}}

\begin{proposition}\label{prop:rank} If $\mathfrak{h}^c$ is a simple non-regular subalgebra in $\mathfrak{g}^c$ then the inequality 
\begin{equation}\label{eq:rank}
\operatorname{rank}\mathfrak{h}^c\leq\frac{1}{2}(\operatorname{rank}\mathfrak{g}^c+1)
\end{equation}
 holds for all pairs $(\mathfrak{g}^c,\mathfrak{h}^c)$ except the following:
\begin{enumerate}
\item $\mathfrak{h}^c=B_{p-1}\subset D_p$,
\item $\mathfrak{h}^c\subset B_{p-1}\subset D_p\subset B_n$, where $D_p\subset B_n$ is regular,
\item $\mathfrak{h}^c\subset B_{p-1}\subset D_p\subset D_n$, where $D_p\subset D_n$ is regular.
\end{enumerate}
\end{proposition}
\begin{proof} Consider first the case when $\mathfrak{h}^c$ is irreducible. If $\mathfrak{h}^c\hookrightarrow \mathfrak{g}^c$ is a fundamental representation, Table \ref{tttab1} shows that (\ref{eq:rank}) does hold (with exceptions 1)-3)). Weyl dimension formula shows that the same is valid for any irreducible representation. If $\mathfrak{h}^c$ is reducible, one uses Theorem \ref{thm:max-ov} and obtains an inclusion $\mathfrak{h}^c\hookrightarrow\tilde{\mathfrak{h}}^c_1$, where $\tilde{\mathfrak{h}}^c_1$ denotes the simple summand of some maximal subalgebra $\tilde{\mathfrak{h}}^c$ (one can also consult Table \ref{tab-max}). Since $\operatorname{rank}\tilde{\mathfrak{h}}^c_1<\operatorname{rank}\mathfrak{g}^c$, one can argue as follows. If $\mathfrak{h}^c\subset\tilde{\mathfrak{h}}^c$ where $\tilde{\mathfrak{h}}^c$ is parabolic, then $\mathfrak{h}^c$ is not regular in $\tilde{\mathfrak{h}}^c$ and one can argue by  induction. If $\tilde{\mathfrak{h}}^c$ is symmetric, then the following two possibilities may occur: either $\mathfrak{h}^c$ is still regular in $\mathfrak{g}^c$, and we apply already proved part of Theorem \ref{thm:class}, or it is one of the exceptions 1)-3). If $\mathfrak{h}^c\subset\tilde{\mathfrak{h}}^c$ is a tensor representation, then the difference between $\operatorname{rank}\mathfrak{h}^c$ and $\mathfrak{g}^c$ is even bigger.
\end{proof}
\begin{proposition}\label{prop:inequality} If $\mathfrak{h}^c\hookrightarrow\mathfrak{g}^c$ and $\mathfrak{l}^c\hookrightarrow\mathfrak{g}^c$ are embeddings of simple and non-regular $\mathfrak{h}^c$ and $\mathfrak{l}^c$, then for any real forms $\mathfrak{h}\hookrightarrow\mathfrak{g}$ and $\mathfrak{l}\hookrightarrow\mathfrak{g}$ either  the inequality $d(\mathfrak{g})>d(\mathfrak{h})+d(\mathfrak{l})$ holds (and hence the triple $(\mathfrak{g},\mathfrak{h},\mathfrak{l})$ cannot yield compact Clifford-Klein forms), or $\mathfrak{h}^c=B_{p-1}$ is embedded into $\mathfrak{g}$ as in cases 1)-3) in Proposition \ref{prop:rank}.
However, none of the embeddings 1)-3) above can yield compact Clifford-Klein forms. 
\end{proposition}
\begin{proof} The proof of the first part of this  Proposition is  a rather complicated case-by case calculation of non-compact dimensions and  we postpone it until  subsection \ref{liczenie}. Here we settle the cases of embeddings 1)-3) and explain the method of computing $d(\mathfrak{h})$ and $d(\mathfrak{l})$ for the real forms of simple irreducible subalgebras which are not of types 1)-3).
\vskip6pt
\noindent {\it Cases 1)-3)}. 
First, consider the embedding 
$$\mathfrak{h}^c\subset B_n\subset D_{n+1},$$
where $\mathfrak{h}^c$ is a regular simple subalgebra of $B_n$, and $B_n$ is a subalgebra of all fixed points of an outer automorphism of $D_{n+1}$. The pair $(\mathfrak{g},\mathfrak{h})$ cannot generate a standard compact Clifford-Klein form such that $\mathfrak{g}\neq \mathfrak{h}+\mathfrak{l}$ (for some $\mathfrak{l}\subset \mathfrak{g}$).
Choose a Cartan subalgebra $\mathfrak{c}\subset D_{n+1}$ such that $\mathfrak{c}\cap\mathfrak{h}^{c}$ is a Cartan subalgebra of $\mathfrak{h}^{c}$. The relation between roots and root vectors for the embedding $B_n\subset D_{n+1}$ is well known. Following \cite{he}, p. 507, we write $H_1,...,H_{n+1}$ for the simple coroots and the root vectors corresponding to them as $X_1,...,X_{n+1},$. The opposite root vectors are denoted by $Y_1,...,Y_{n+1}$. The following vectors determine the embedding of $B_n$ into $D_{n+1}$:
$$H_1,...,H_{n-1},H_n+H_{n+1},$$
$$ X_1,...,X_{n-1},X_n+X_{n+1}, Y_1,...,Y_{n-1},Y_n+Y_{n+1}.$$
Since  any subalgebra $\mathfrak{h}^c$ which is regular in $B_n$ is obtained from the extended Dynkin diagram of $B_n$ by deleting some vertices, we can write all the possibilities for the simple roots and root vectors of $\mathfrak{h}^c$ {\it in terms of the simple roots, or corresponding coroots and root vectors expressed via the data for the type $D_{n+1}$.} The extended Dynkin diagram for $B_n$ is given by vertices
$$\alpha_0,\alpha_1,...,\alpha_{n-1},\bar\alpha_{n} : = \frac{1}{2} (\alpha_{n}+\alpha_{n+1})$$
 
\begin{equation}
\alpha_0=-(\alpha_1+2\alpha_2+\cdots+2\alpha_{n-1}+(\alpha_n+\alpha_{n+1})) .
\end{equation}\label{eq:ext}
Notice that $\alpha_{0}$ is a root of $D_{n+1} .$  Moreover, it is the {\it minimal root}  of $D_{n+1}$ (see \cite{bouu}, Table IV). We have shown the following:  the system of simple roots for any $\mathfrak{h}^c$  is obtained by deleting some of the roots from (7)  in a way to get a connected diagram. In particular, if $\mathfrak{h}^c$ is simple and not regular in $\mathfrak{g}^{c} ,$ then $\mathfrak{h}^{c}$ is obtained by deleting $\alpha_0,\alpha_1,...,\alpha_{t},$ $1\leq t<n,$ because otherwise $\mathfrak{h}^c$ would be regular in $D_{n+1}$. Therefore, $\mathfrak{h}^c$ is necessarily of the form $B_p$. Looking at \cite{ov}, Table 4, one can see that there is only one possible real form (up to conjugation):
$$\mathfrak{h}= \mathfrak{so}(q,2p+1-q).$$

Let $\theta$ be a Cartan involution of $\mathfrak{g}$ such that $\theta (\mathfrak{h}) = \mathfrak{h} .$ We may assume that $\theta (\mathfrak{c})=\mathfrak{c} .$ It follows from the classification of involutive automorphisms (see Section 1.4, Chapter 4 in \cite{ov}) and the labels of the extended Dynkin diagram of $D_{n+1}$ that $\theta = \operatorname{Ad} I_{m,2n+2-m}$ for some $p\in \{ 2,...,n+1 \} $ (\cite{ov}, Theorem 1.4 (2), Chapter 4). Indeed, the Kac diagram of $\theta$ cannot be given by two $1's$ on  the opposite sides of the extended Dynkin diagram of $D_{n+1}$ (because $\theta (\mathfrak{h}^{c}) = \mathfrak{h}^{c}$).
Therefore, 
$$\mathfrak{h}=\mathfrak{so}(q, 2p+1-q)\subset \mathfrak{g}=\mathfrak{so}(m,2n+2-m),$$
 where $q<m\leq n+1.$
 By Theorem \ref{thm:kob-d} we have $2p+1-q=2n+2-m.$ Indeed, if $2p+1-q<2n+2-m$, one can construct $\mathfrak{h}\subset\mathfrak{g}'\subset\mathfrak{g}$ with $\mathfrak{a}_{g'}=\mathfrak{a}_h$ and $d(\mathfrak{g}')>d(\mathfrak{h})$ by increasing $2p+1-q$ but still keeping it less than $2n+2-m$. One can do this analyzing separately all the possible choices of $q$ and $m$ even or odd, and using the formulas for the non-compact dimensions from Table 2. The inequality $2p+1-q>2n+2-m$ is not possible for dimensional reasons.
This yields   an embedding
$$\mathfrak{so}(q,2n+2-m)\subset \mathfrak{so}(m,2n+2-m).$$
Since $\theta(\mathfrak{h})\subset\mathfrak{h}$, we get the inclusion of maximal compact subalgebras of $\mathfrak{h}$ and $\mathfrak{g}$:
$$\mathfrak{so}(q)\oplus\mathfrak{so}(2n+2-m)\subset\mathfrak{so}(m)\oplus\mathfrak{so}(2n+2-m).$$
Therefore there is an obvious inclusion
$$\mathfrak{so}(m-q)\oplus \mathfrak{so}(q,2n+2-m)\subset \mathfrak{so}(m,2n+2-m) .$$
We see that the pair $(\mathfrak{so}(m,2n+2-m) , \mathfrak{so}(m-q)\oplus \mathfrak{so}(q,2n+2-m) )$ is symmetric and $(\mathfrak{so}(m,2n+2-m) , \mathfrak{so}(m-q)\otimes \mathfrak{so}(q,2n+2-m) )$ induces a compact Clifford-Klein form if and only if $(\mathfrak{so}(m,2n+2-m) , \mathfrak{so}(q,2n+2-m) )$ induces a compact Clifford-Klein form. Therefore, \cite{tojo} yields the desired result.

Now we are ready to explain the method of calculating the non-compact dimensions for all  embeddings of simple non-regular Lie subalgebras. 
\begin{enumerate}
\item By Proposition \ref{prop:rank} the inequality (\ref{eq:rank}) holds.
\item looking at Tables \ref{tttab2},\ref{tttab1},\ref{tttttab2}, \ref{tttab3} and using relations between real and complex ranks together with the constraint given by inequality (\ref{eq:rank}), one obtains the desired inequality for the non-compact dimensions.

\end{enumerate}  
As we have mentioned, the details of the calculation are postponed until Subsection \ref{liczenie}.

\end{proof}
Clearly, Proposition \ref{prop:inequality} exhaust all the possibilities for non-regular triples $(\mathfrak{g}^c,\mathfrak{h}^c,\mathfrak{l}^c)$ of simple complex Lie algebras and their real forms, which completes the proof of Theorem \ref{thm:class}.

\subsubsection{Proof of Proposition \ref{prop:inequality}}\label{liczenie}

 We need to list all simple real forms of non-compact type from \cite{ov}, Table 4, in the form suitable for the calculations of the non-compact dimensions. By a slight abuse of  notation we will treat simple but not absolutely simple real Lie algebras of rank 2n as Lie algebras of type $n.$ For instance $\mathfrak{g}=\mathfrak{sl}(n+1,\mathbb{C})$ is classified as type $A_n .$  Notice that in this case $\mathfrak{g}^{c} = \mathfrak{sl}(n+1,\mathbb{C}) \oplus \mathfrak{sl}(n+1,\mathbb{C}) .$
\vskip6pt
\centerline{\bf Simple real forms of non-compact type}
\vskip6pt
\textbf{$A_{t}$}: 
\begin{itemize}
	\item $\mathfrak{su} (a,t+1-a) ,$ $a\leq \frac{t+1}{2},$ $d:=d(\mathfrak{g}) = 2at+2a-2a^{2} ,$ $r:= \textrm{rank}_{\mathbb{R}}(\mathfrak{g})=a,$
	\item $\mathfrak{su}^{\ast} (t+1) ,$ $\frac{t-1}{2}\in \mathbb{N},$ $d= \frac{1}{2} (t^{2}+t-2),$ $r=\frac{t-1}{2},$
	\item $\mathfrak{sl} (t+1,\mathbb{R}) ,$  $d= \frac{1}{2} t(t+3),$ $r=t,$
	\item $\mathfrak{sl} (t+1,\mathbb{C}) ,$ $d= t^{2}+2t,$ $r=t.$
\end{itemize}

\textbf{$B_{t}$}:
\begin{itemize}
	\item $\mathfrak{so} (a,2t+1-a) ,$ $a<t+1,$ $d= 2at+a-a^{2},$ $r=a,$
	\item $\mathfrak{so} (2t+1,\mathbb{C}) ,$ $d= t(2t+1),$ $r=t.$
\end{itemize}

\textbf{$C_{t}$}:
\begin{itemize}
	\item $\mathfrak{sp} (a,t-a) ,$ $a\leq\frac{t}{2},$ $d= 4at-4a^{2},$ $r=a,$
	\item $\mathfrak{sp} (t,\mathbb{R}) ,$  $d= t^{2}+t,$ $r=t,$
	\item $\mathfrak{sp} (t,\mathbb{C}) ,$  $d= 2t^{2}+t,$ $r=t,$
\end{itemize}

\textbf{$D_{t}$}:
\begin{itemize}
	\item $\mathfrak{so} (a,2t-a) ,$ $a<t+1,$ $d= 2at-a^{2},$ $r=a,$
	\item $\mathfrak{so}^{\ast} (t) ,$ $d= t^{2}-t,$ $r=[\frac{t}{2}],$
	\item $\mathfrak{so} (2t,\mathbb{C}) ,$  $d= 2t^{2}-t,$ $r=t.$
\end{itemize}

\textbf{$G_{2}$} (simplest representation $G_{2}\subset B_{3}$):
\begin{itemize}
	\item $\mathfrak{g}_{2(2)} ,$  $d= 8,$ $r=2,$ 
	\item $\mathfrak{g}_{2}(\mathbb{C}) ,$  $d= 14,$ $r=2,$ 
\end{itemize}

\textbf{$F_{4}$} (simplest representation $F_{4}\subset D_{13}$):
\begin{itemize}
	\item $\mathfrak{f}_{4(4)} ,$  $d= 28,$ $r=4,$ 
	\item $\mathfrak{f}_{4(-2)} ,$  $d= 16,$ $r=1,$ 
	\item $\mathfrak{f}_{4}(\mathbb{C}) ,$  $d= 52,$ $r=4,$ 
\end{itemize}

\textbf{$E_{6}$} (simplest representation $E_{6}\subset A_{26}$):
\begin{itemize}
	\item $\mathfrak{e}_{6(6)} ,$  $d= 42,$ $r=6,$ 
	\item $\mathfrak{e}_{6(2)} ,$  $d= 40,$ $r=4,$ 
	\item $\mathfrak{e}_{6(-14)} ,$  $d= 32,$ $r=2,$ 
	\item $\mathfrak{e}_{6(-26)} ,$  $d= 26,$ $r=2,$ 
	\item $\mathfrak{e}_{6}(\mathbb{C}) ,$  $d= 78,$ $r=6,$ 
\end{itemize}

\textbf{$E_{7}$} (simplest representation $E_{7}\subset C_{56}$):
\begin{itemize}
	\item $\mathfrak{e}_{7(7)} ,$  $d= 70,$ $r=7,$ 
	\item $\mathfrak{e}_{7(-5)} ,$  $d= 64,$ $r=4,$ 
	\item $\mathfrak{e}_{7(-25)} ,$  $d= 54,$ $r=3,$ 
	\item $\mathfrak{e}_{7}(\mathbb{C}) ,$  $d= 133,$ $r=7,$ 
\end{itemize}

\textbf{$E_{8}$} (simplest representation $E_{8}\subset D_{124}$):
\begin{itemize}
	\item $\mathfrak{e}_{8(8)} ,$  $d= 128,$ $r=8,$ 
	\item $\mathfrak{e}_{8(-24)} ,$  $d= 112,$ $r=4,$
	\item $\mathfrak{e}_{8}(\mathbb{C}) ,$  $d= 248,$ $r=8,$ 
\end{itemize}

Now we make some extra analysis of this list.
\vskip6pt
\centerline{\bf Observations based on the list above}
\vskip6pt
\begin{enumerate}
	\item For all classical types with $\textrm{rank}=t$ and $\textrm{rank}_{\mathbb{R}}=a$ one of the following three algebras has the highest number $d:=d(\mathfrak{g})$
	 \begin{itemize}
		\item I$:=\mathfrak{so} (a,2t+1-a) ,$ $d= 2at+a-a^{2},$
		\item II$:\mathfrak{sp} (a,t-a) ,$ $d= 4at-4a^{2},$
		\item III$:=\mathfrak{sp} (t,\mathbb{C}) ,$  $d= 2t^{2}+t.$ 
	 \end{itemize}
	\item For each real form of type $F_{4},E_{6},E_{7},E_{8}$ one can find a real form of type I or II with the same real rank, with greater or equal $d$ and with rank at most half of the rank of the simplest representation. For instance consider type $F_{4}.$ 
	 \begin{itemize}
		\item $\mathfrak{f}_{4(4)}$ ($d=28,$ $r=4$) has the simplest representation in $D_{13}$ (after complexification). One can substitute $\mathfrak{f}_{4(4)}$ with $\mathfrak{so}(4,9)$ ($d=36,$ $r=4$) which is of the type $B_{6}.$
		\item $\mathfrak{f}_{4(-2)}$ ($d=16,$ $r=1$) has the simplest representation in $D_{13}$ (after complexification). One can substitute $\mathfrak{f}_{4(-2)}$ with $\mathfrak{sp}(1,5)$ ($d=20,$ $r=1$) which is of the type $C_{6}.$
		\item $\mathfrak{f}_{4}(\mathbb{C})$ ($d=52,$ $r=4$) has the simplest representation in $D_{13}\oplus D_{13}\subset D_{26}$ (after complexification). One can substitute $\mathfrak{f}_{4}(\mathbb{C})$ with $\mathfrak{so}(4,21)$ ($d=83,$ $r=4$) which is of the type $B_{12}.$
	 \end{itemize}
	\item For $\mathfrak{g}_{2}(\mathbb{C})$ one can find a real form of type I ($\mathfrak{so}(2,7)$) with the same real rank, and  $d$ the same or greater, and with rank {\it at most half of the rank} of the representation into an algebra of rank $8$ or greater. As the simplest representation of $\mathfrak{g}_{2}(\mathbb{C})$ is $B_{6},$  in each case one has to verify separately the embeddings of $\mathfrak{g}_{2}(\mathbb{C})$ into some Lie algebras of rank $6$ and $7.$
	\item For $\mathfrak{g}_{2(2)}$ one can find a real form of type I ($\mathfrak{so}(2,5)$) with the same real rank, with  $d$ the same or greater, and with rank  {\it at most half of the rank} of the representation into a Lie  algebra of rank 6 or greater. As the simplest representation of $\mathfrak{g}_{2(2)}$ is $B_{3}$, one should  verify the embeddings of $\mathfrak{g}_{2(2)}$ into Lie algebras of rank $4$ and $5.$ This is because the real rank of $\mathfrak{g}_{2(2)}$ equals 2 and so if $(\mathfrak{g}, \mathfrak{g}_{2(2)}, \mathfrak{l})$ gives a standard compact Clifford-Klein form then necessarily the real rank of $\mathfrak{g}$ is at least 3. If the $\textrm{rank}(\mathfrak{g})=3$ then $\mathfrak{g}$ is split and we can refer to Corollary \ref{cor:split}.
\end{enumerate}

We only have to check the triples $(\mathfrak{g}, \mathfrak{h}, \mathfrak{l})$ such that $\mathfrak{h}$ and $\mathfrak{l}$ are one of the following
\begin{itemize}
	\item absolutely simple non-regular,
	\item simple but not absolutely simple,
\end{itemize}
and $\textrm{rank}_{\mathbb{R}}(\mathfrak{h})=:a>0,$ $\textrm{rank}_{\mathbb{R}}(\mathfrak{l}) =: b>0,$ $\textrm{rank}_{\mathbb{R}}(\mathfrak{g})=a+b.$
\vskip6pt
\centerline{\bf  Final check up of the inequality $d(\mathfrak{h})+d(\mathfrak{l})<d(\mathfrak{g})$}
\vskip6pt
By the observations above we only have to check the inequality for types I,II,III (and $G_{2}$ for low ranks).

By $X_{t}$ we will denote an arbitrary type of rank $t.$ By $(i,j),$ $i,j\in \{ I,II,III  \} $ we denote a pair of subalgebras of type i and j, respectively. We exclude the cases in which $\mathfrak{g}$ is split,  because of \cite{botr}. 

\noindent {\bf Type $A_{2n}$}
\vskip6pt
\noindent {\bf Case $\mathfrak{g}=\mathfrak{su}(a+b,2n+1-a-b)$}
\vskip6pt
We begin with $\mathfrak{h}$ of type $G_{2}.$ First we have to check $\mathfrak{g}_{2(2)}$ in $A_{4}.$ But in this case the real rank of $\mathfrak{g}$ is at most 2. Secondly we have $\mathfrak{g}_{2}(\mathbb{C})$ in $A_{6}.$ So necessarily $\mathfrak{g}=\mathfrak{su}(3,4).$ One easily checks that there is no simple $\mathfrak{l}$ of real rank one with $d(\mathfrak{l})>8.$ 

Clearly, the only proper regular subalgebras in $A_{2n}$ are of type $A_{p},$ $p<2n.$ Note that any non-regular simple Lie subalgebra of $A_{2n}$ is of the rank at most $n,$ by Proposition \ref{prop:inequality}. 

 If $\mathfrak{h}$ is simple but not absolutely simple then it is obviously of type $X_{p}$ with $p\leq n.$ Thus we have to check the Lie subalgebras of type I,II and III for $t=n.$ We have
\begin{itemize}
	\item $d(\mathfrak{g}) = 4an+4bn+2a+2b-2a^{2} - 2b^2 -4ab,$ $r=a+b,$
	\item $d(I) = 2an+a-a^2,$ $r=a,$
	\item $d(II) = 4an-4a^2,$ $r=a,$
	\item $d(III) = 2a^2 +a,$ $r=a.$
\end{itemize}
We also have $a+b\leq n.$

\textbf{(I,I)} $D=4an+4bn+2a+2b-2a^{2} - 2b^2 -4ab - 2an - a +a^2 - 2bn-b + b^2\geq a^2 + b^2 +a +b>0,$

\textbf{(I,II)} $D=4an+4bn+2a+2b-2a^{2} - 2b^2 -4ab- 2an - a +a^2 -4bn + 4b^2 \geq (a-b)^2 +a+2b+ b^2 >0,$

\textbf{(I,III)} $D=4an+4bn+2a+2b-2a^{2} - 2b^2 -4ab- 2an - a +a^2 -2 b^2 -b \geq a^2 +2ab +a+b>0,$

\textbf{(II,II)} $D=4an+4bn+2a+2b-2a^{2} - 2b^2 -4ab - 4an+4 a^2 -4bn +4 b^2 = 2(a-b)^2 +2a+2b>0,$

\textbf{(II,III)} $D=4an+4bn+2a+2b-2a^{2} - 2b^2 -4ab - 4an+4 a^2 -2b^2 -b \geq 2 a^2 +2a +b >0 ,$

\textbf{(III,III)} $D=4an+4bn+2a+2b-2a^{2} - 2b^2 -4ab -2 a^2 - a - 2 b^2 - b \geq 4ab+a+b>0 .$

\noindent {\bf Type $A_{2n+1}$}
\vskip6pt

Again, by Proposition \ref{prop:inequality} we see that if $X_p\subset A_{2n+1}$ is a classical simple non-regular subalgebra or simple but not absolutely simple subalgebra then $p\leq n+1.$

We begin with subalgebras of type $G_{2}. $ 
\begin{itemize}
	\item If $\mathfrak{h}:=\mathfrak{g}_{2(2)} \subset A_{5}$ then (since the real rank of $\mathfrak{g}$ has to be greater than 2 and we omit the split case) necessarily $\mathfrak{g}=\mathfrak{su}(3,3).$ But $\mathfrak{g}_{2(2)}$ is not a subalgebra of $\mathfrak{g} .$ 
	\item If $\mathfrak{h}:=\mathfrak{g}_{2}(\mathbb{C}) \subset A_{7}$ then (since the real rank of $\mathfrak{g}$ has to be greater than 2 and we omit the split case) necessarily $\mathfrak{g}=\mathfrak{su}(3,5),$  $\mathfrak{su}(4,4), $ $ \mathfrak{su}^{\ast}(8).$ But the maximal compact subgroup of $\mathfrak{g}_{2}(\mathbb{C})$ is $\mathfrak{g}_{2} .$ The maximal compact subgroups of possible $\mathfrak{g}$ are $\mathfrak{su}(3)\oplus \mathfrak{u}(5), $ $\mathfrak{su}(4)\oplus \mathfrak{u}(4),$ $\mathfrak{sp}(4) $ and so $\mathfrak{g}_{2(2)}$ is not a subgroup of $\mathfrak{g} .$ 
\end{itemize}

\noindent {\bf Case $\mathfrak{g}=\mathfrak{su}^{\ast}(2n+2)$}
\vskip6pt
In this case we have to check Lie subalgebras of type I,II and III for $t=n+1.$ We have
\begin{itemize}
	\item $d(\mathfrak{g}) = 2 a^2 +2 b^2 +4ab +3a+3b,$ $r=a+b=n$
	\item $d(I) = 2an+3a-a^2,$ $r=a,$
	\item $d(II) = 4an-4a^2+a,$ $r=a,$
	\item $d(III) = 2a^2 +a,$ $r=a.$
\end{itemize}
We also have $a+b=n.$

\textbf{(I,I)} $D= 2 a^2 +2 b^2 +4ab +3a +3b -2a(a+b) + a^2 -3a -2b(a+b) + b^2 -3b = a^2 + b^2 >0,$

\textbf{(I,II)} $D= 2 a^2 +2 b^2 +4ab +3a +3b -2a(a+b) + a^2 -3a -4b(a+b) + 4 b^2 -b = (a-b)^2 + b^2 +2b>0,$

\textbf{(I,III)} $D= 2 a^2 +2 b^2 +4ab +3a +3b -2a(a+b) + a^2 -3a -2b^2 -b = a^2 +2ab +2b >0,$

\textbf{(II,II)} $D= 2 a^2 +2 b^2 +4ab +3a +3b -4a(a+b) + 4 a^2 -a -4b(a+b) + 4 b^2 -b = 2(a-b)^2 +2(a+b) >0,$

\textbf{(II,III)} $D= 2 a^2 +2 b^2 +4ab +3a +3b -4a(a+b) + 4 a^2 -a  -2b^2 -b = 2a^2 +2(a+b) >0,$

\textbf{(III,III)} $D= 2 a^2 +2 b^2 +4ab +3a +3b  -2a^2 -a -2b^2 -b= 4ab +2(a+b)  >0.$

\noindent {\bf Case$\mathfrak{g}=\mathfrak{su}(a+b,2n+2-a-b)$}
\vskip6pt
In this case we have to check Lie subalgebras of type I,II for $t=n+1.$ We can substitute III by III'$:=\mathfrak{sl}(a+1,\mathbb{C}),$ $a\leq n$ .  We have
\begin{itemize}
	\item $d(\mathfrak{g}) = 4an+4bn -2 a^2 - 2 b^2 -4ab +4a+4b,$ 
	\item $d(I) = 2an+3a-a^2,$ $r=a,$
	\item $d(II) = 4an-4a^2+a,$ $r=a,$
	\item $d(III') = a^2 +2a,$ $r=a.$
\end{itemize}
We have $a+b\leq n+1,$ $n\geq 2.$

\textbf{(I,I)} $D=4an+4bn -2 a^2 - 2 b^2 -4ab +4a+4b -2an-3a+a^2  -2bn-3b+b^2 \geq a(a-1) + b(b-1) >0,$ unless $a=b=1.$ But if $a=b=1$ then $D=4(n-1)>0,$

\textbf{(I,II)} $D=4an+4bn -2 a^2 - 2 b^2 -4ab +4a+4b -2an-3a+a^2   -4bn+4 b^2 -b \geq (a-b)(a-b-1) +b^2 +2b >0,$

\textbf{(I,III')} $D=4an+4bn -2 a^2 - 2 b^2 -4ab +4a+4b -2an-3a+a^2 -b^2 -2b \geq (a-1)(a+b) +ab +b(b-1)  >0,$

\textbf{(II,II)} $D=4an+4bn -2 a^2 - 2 b^2 -4ab +4a+4b  -4an+4 a^2 -a  -4bn+4 b^2 -b = 2(a-b)^2 +3(a+b) >0, $

\textbf{(II,III')} $D=4an+4bn -2 a^2 - 2 b^2 -4ab +4a+4b  -4an+4 a^2 -a -b^2 -2b =b(b-2) +2 a^2 +3a  >0, $  

\textbf{(III',III')} $D=4an+4bn -2 a^2 - 2 b^2 -4ab +4a+4b  -a^2 -2a -b^2 -2b \geq (a+b)(a+b-2)+2ab >0.$

\vskip6pt
\noindent {\bf Type $C_{n}$}

In this case we only have to check $\mathfrak{g} = \mathfrak{sp}(a+b,n-a-b) .$

We begin with subalgebras of type $G_{2}. $ We can eliminate (in the case $\mathfrak{h}=\mathfrak{g}_{2(2)}$) $C_{4}$ and $C_{5}$ since the real rank of $\mathfrak{sp}(2,2)$ and $\mathfrak{sp}(2,3)$ equals $2.$ If $\mathfrak{h}:=\mathfrak{g}_{2}(\mathbb{C}) \subset A_{6}$ then (since the real rank of $\mathfrak{g}$ has to be greater than 2 and we omit the split case) necessarily $\mathfrak{g}=\mathfrak{sp}(3,3).$ But now $\mathfrak{l}\subset \mathfrak{sp}(3,3)$ has to have real rank equal to 1 and $d=22.$ One easily verifies that this is impossible. We can use the same argument for $C_{7}$ (in this case necessarily $\mathfrak{g}=\mathfrak{sp}(3,4)$).

Regular proper subalgebras of $C_n$ are given by $A_{p}, C_{p}, D_{s},$ $p+1,s \leq n.$ By Proposition \ref{prop:inequality} to obtain non-regular subalgebra in $C_{n}$ of rank greater than $\frac{n}{2}$ one has to ``use'' the non-regular embedding $B_{p-1}\subset D_{p},$ where $D_{p}$ is regular in $C_{n} .$ By Proposition \ref{prop:inequality}  any non-regular simple subalgebra in $D_{p}$ of rank greater or equal to $\frac{p}{2}$ is of the type $B.$ Thus we have to check the following embeddings:
$$A_{\frac{n}{2}}, B_{n-1}, C_{{\frac{n}{2}}}, D_{{\frac{n}{2}}}\subset C_{n},$$
and so we have
\begin{itemize}
	\item $d(\mathfrak{g}) = 4an+4bn-4 a^2 - 4 b^2 -8ab,$ 
	\item $d(I) = 2an-a-a^2,$ $r=a,$
	\item $d(II) = 2an-4a^2,$ $r=a,$
	\item $d(III) = 2a^2 +a,$ $r=a,$ $a\leq \frac{n}{2}$
\end{itemize}
We have $a+b\leq\frac{n}{2},$ $n\geq 2.$ Notice that we only have to check I and III.

\textbf{(I,I)} $D= 4an+4bn-4 a^2 - 4 b^2 -8ab -  2an+a+a^2 -  2bn+b+b^2 \geq a^2 + b^2 +a+b >0,$

\textbf{(I,III)} $D= 4an+4bn-4 a^2 - 4 b^2 -8ab -  2an+a+a^2 -2b^2 - b \geq a^2 +4ab +a+b(2b-1)>0, $

\textbf{(III,III)} $D= 4an+4bn-4 a^2 - 4 b^2 -8ab -2a^2-a  -2b^2 - b  \geq 2(a^2 + b^2 ) +a(4b-1)+b(4a-1)  >0.$
  
\vskip6pt

\noindent {\bf Type $B_{n}$}
\vskip6pt

In this type we only have to check $\mathfrak{g} = \mathfrak{so}(a+b,2n+1-a-b) .$

We begin with subalgebras of type $G_{2}. $ If $\mathfrak{h}=\mathfrak{g}_{2(2)}$ then $\mathfrak{g}=\mathfrak{so}(3,6),$ $\mathfrak{so}(4,5).$ If $\mathfrak{h}=\mathfrak{g}_{2}(\mathbb{C})$ then $\mathfrak{g}=\mathfrak{so}(s,13-s) , \mathfrak{so}(s,15-s),$ $s=3,4,5,6,7.$ One verifies that there does not exist $\mathfrak{l}$ with appropriate $r,d$, with one exception:
$$(\mathfrak{so}(6,7),\mathfrak{so}(4,7),\mathfrak{g}_{2}(\mathbb{C})),$$
but in this case $\mathfrak{g}$ is split.

Regular proper subalgebras of $B_{n}$ are given by $A_{p}, B_{p}, D_{s},$ $p+1,s \leq n.$ By Proposition \ref{prop:inequality} to obtain non-regular subalgebra in $B_{n}$ of rank greater than $\frac{n}{2}$ one has to consider the chain

 $$\mathfrak{h}^c\subset B_{p-1}\subset D_{p}\subset B_n$$
 where $D_{p}$ is regular in $B_{n}.$ In such case no real form of the pair $(\mathfrak{g}^c,\mathfrak{h}^c)$ can yield a compact standard Clifford-Klein form. Thus we may assume that each subalgebra is of rank at most $\frac{n}{2} $.

In this case we have to check Lie subalgebras of type I,II,III for $t=\frac{n}{2}.$ 
\begin{itemize}
	\item $d(\mathfrak{g}) = 2an+2bn - a^2 -  b^2 -2ab +a+b,$ 
	\item $d(I) = an+a-a^2,$ $r=a,$
	\item $d(II) = 2an-4a^2,$ $r=a,$
	\item $d(III) = 2a^2 +a,$ $r=a,$ $a\leq \frac{n}{2}.$
\end{itemize}
We have $a+b\leq n.$

\textbf{(I,I)} $D= 2an+2bn - a^2 -  b^2 -2ab +a+b - an-a+a^2  - bn-b+b^2\geq a^2 + b^2 >0.$

\textbf{(I,II)} $D= 2an+2bn - a^2 -  b^2 -2ab +a+b - an-a+a^2 -2bn+4b^2 \geq (a-b)^2 +ab + b + 2b^2    >0,$

\textbf{(I,III)} $D= 2an+2bn - a^2 -  b^2 -2ab +a+b - an-a+a^2 -2b^2-b \geq (a-b)^2 +ab   >0,$

\textbf{(II,II)} $D= 2an+2bn - a^2 -  b^2 -2ab +a+b  -2an+4a^2 -2bn+4b^2  \geq (a-b)^2 + 2(a^2 + b^2 ) +a+b   >0,$

\textbf{(II,III)} $D= 2an+2bn - a^2 -  b^2 -2ab +a+b  -2an+4a^2   -2b^2- b \geq (a-b)^2 2a^2 +a+2b    >0,$

\textbf{(III,III)} $D= 2an+2bn - a^2 -  b^2 -2ab +a+b  -2a^2- a  -2b^2- b \geq (a-b)^2 .$ Thus $D>0$ unless $a=b=\frac{n}{2}.$ But then $\mathfrak{g}=\mathfrak{so}(n,n+1),$ which is split.

\vskip6pt
\noindent {\bf Type $D_{n}$}

Notice that if $\mathfrak{h}=\mathfrak{g}_{2}(\mathbb{C})$ then we have to check $D_{6}$ and $D_{7}.$ But if $\mathfrak{g}_{2}(\mathbb{C})$ is contained in the Lie algebra of type $D_{p}$ then necessarily $p\geq 8.$ 

Regular proper subalgebras of $D_{n}$ are given by $A_{p}, D_{p},$ $p \leq n.$  Proposition \ref{prop:inequality} eliminates the possibility $\operatorname{rank}\mathfrak{h}^c>\frac{1}{2} \operatorname{rank}\mathfrak{g}^c$. The opposite inequality is analyzed below.

\vskip6pt
\noindent {\bf Case $\mathfrak{g}=\mathfrak{so}^{\ast} (n)$}
\vskip6pt

We begin with subalgebras of type $G_{2}. $ If $\mathfrak{h}=\mathfrak{g}_{2(2)}$ then we have to check $\mathfrak{g}=\mathfrak{so}^{\ast} (4), \mathfrak{so}^{\ast} (5) .$ In both cases real rank equals 2 (which equals the real rank of $\mathfrak{g}_{2(2)}$). 

By Proposition \ref{prop:inequality} any non-regular simple subalgebra in $B_{n}$ of rank greater or equal to $\frac{p}{2}$ is of the type $B.$ Thus we have to check the following embeddings:
$$A_{\frac{n}{2}}, B_{n-1}, C_{{\frac{n}{2}}}, D_{{\frac{n}{2}}}\subset D_{n},$$
and so we have
\begin{itemize}
	\item $d(\mathfrak{g}) = n^2 -n,$ 
	\item $d(I) = 2an-a-a^2,$ $r=a,$
	\item $d(II) = 2an-4a^2,$ $r=a,$
	\item $d(III) = 2a^2 +a,$ $r=a,$ $a\leq \frac{n}{2} .$
\end{itemize}
We see that one only needs to check I and III. Also $n\geq 4,$ $\frac{n-1}{2}\leq a+b \leq \frac{n}{2}$ and so $a^2 + b^2 \leq \frac{n^2}{4}-2ab.$
\textbf{(I,I)} $D=   n^2 -n - 2an+a+a^2 - 2bn+b+b^2 \geq -n + a^2 +a +b^2 +b>0,$ unless $a=b=1$ and $n=4.$ But in this case we have embedding $B_{3}\subset D_{4},$ which is symmetric and we can apply \cite{tojo}.

\textbf{(I,III)} $D=   n^2 -n - 2an+a+a^2 - 2b^2 -b \geq 2bn-n +a + a^2 -2b^2 - b \geq 2b^2 - 3b +4ab +a^2 -a -1 =b(2b-3) +2a +a(a-1) + 2ab -1  >0,$

\textbf{(III,III)} $D=   n^2 -n - 2a^2-a-2b^2-b \geq \frac{n^2}{2} - n +2ab -a-b \geq \frac{n}{2}(n-3) +2ab >0.$

\vskip6pt
\noindent {\bf Case $\mathfrak{g}=\mathfrak{so} (a+b,2n-a-b)$}

We begin with subalgebras of type $G_{2}. $ If $\mathfrak{h}=\mathfrak{g}_{2(2)}$ then we have to check $\mathfrak{g}=\mathfrak{so}(3,5),$ $\mathfrak{so}(3,7),$ $\mathfrak{so}(4,6) .$ One verifies that there does not exist $\mathfrak{l}$ with appropriate $r,d.$  

We have the non-regular embedding $B_{p-1}\subset D_{p},$ where $D_{p}$ is regular in $D_{n}.$ By Proposition \ref{prop:inequality} we can eliminate this case. Thus we may assume that each subalgebra of $\mathfrak{so} (a+b,2n-a-b)$ is of rank at most $\frac{n}{2} .$ 

We can substitute III by III'$:=\mathfrak{so}(2a,\mathbb{C}), $ $ \mathfrak{so}(2a+1,\mathbb{C}),$ $a\leq \frac{n}{2}$ . We have
\begin{itemize}
	\item $d(\mathfrak{g}) = 2an+2bn - a^2 -  b^2 -2ab ,$ 
	\item $d(I) = an-a-a^2,$ $r=a,$
	\item $d(II) = 2an-4a^2,$ $r=a,$
	\item $d(III') \leq 2a^2 +a,$ $r=a,$ $a\leq \frac{n}{2} .$
\end{itemize}
We have $a+b\leq n.$

\textbf{(I,I)} $D= 2an+2bn - a^2 -  b^2 -2ab - an+a+a^2 - bn+b+b^2 \geq a^2 + b^2 -a-b  >0,$ unless $a=b=1.$ In such case $D=2n-4>0,$

\textbf{(I,II)} $D= 2an+2bn - a^2 -  b^2 -2ab - an+a+a^2  - 2bn-4b^2 \geq (a-b)^2 = 3b^2 +a(b-1) >0,$

\textbf{(I,III')} $D= 2an+2bn - a^2 -  b^2 -2ab - an+a+a^2  -2b^2 -b \geq (a-b)^2 + a(\frac{1}{2}b -1)  + b(\frac{1}{2}a -1)  >0,$ unless $(a,b)=(1,1),$ $(1,2), $ $(2,1),$ $(2,2).$ As $n\geq 4$ thus $D>0$ unless $n=4,$ $a=b=2.$ In such case we have $\mathfrak{g}=\mathfrak{so}(4,4),$  $\mathfrak{h}=\mathfrak{so}(2,3),$  $\mathfrak{l}=\mathfrak{so}(4,\mathbb{C}),$ and so $D=4>0,$

\textbf{(II,II)} $D= 2an+2bn - a^2 -  b^2 -2ab - 2an+4a^2  - 2bn+4b^2 = (a-b)^2 + 2(a^2 + b^2)>0,$

\textbf{(II,III')} $D= 2an+2bn - a^2 -  b^2 -2ab - 2an+4a^2 -2b^2 -b\geq (a-b)(a-b+1) +a^2   >0,$

\textbf{(III',III')} Notice that for the case III we have one of the following
  \begin{enumerate}
	  \item $d=2a^2+a,$ with $2a+1\leq n,$
		\item $d=2a^2-a,$ with $2a\leq n,$
  \end{enumerate}
	For any combination (that is (III.1,III.1), (III.1,III.2), (III.2,III.2)) we have
	$D\geq (a-b)^2 +a+b>0.$
	
\section{Tables}\label{sec:addendum}

\begin{center}
 \begin{table}[h]
 \centering
 {\footnotesize
 \begin{tabular}{| c | c | c | c | }
\hline 
  \multicolumn{4}{|c|}{ \textbf{\textit{Non-compact dimensions}}} \\
		\hline
 $\text{Type}$ & $d(\mathfrak{g})$ & $p\,\text{via}\,l$ & $\text{restr. on}\,p$   \\
\hline
		$A_t$ & $2a(t+1)-a$ & $$ & $1\leq a\leq \frac{t}{2}$ \\
		\hline
		$A_t$ & $2a^2$ & $t=2a-1$ & $a\geq 2$ \\
		\hline
		$B_t$ & $a(2t+1-a)$ & $$ & $1\leq a\leq t$ \\
		\hline
		$C_t$ & $4a(t-a)$ & $$ & $1\leq a\leq \frac{t}{2}$ \\
		\hline
		$D_t$ & $a(2t-a)$ & $ $ & $1\leq a\leq t$ \\
		\hline
		$D_t$ & $2a(2a-1)$ & $t=2a$ & $ $ \\
		\hline
		$D_t$ & $2a(2a+1)$ & $t=2a+1$ & $ $ \\
		\hline
 \end{tabular}
 }
 \caption{
 }
 \label{tttab2}
 \end{table}
\end{center}

\begin{center}
 \begin{table}[h]
 \centering
 {\footnotesize
 \begin{tabular}{| c | c | c| }
   \hline
   \multicolumn{3}{|c|}{ \textbf{\textit{Dimensions of fundamental representations}}} \\
   \hline                        
   $\text{Type}$ & $\text{dimension}$ & $\bar\omega_r$   \\
   \hline
		$A_l$ & $\binom{l+1}{r}$ & $0,r\not=(l+1)/2,(-1)^r, r=(l+1)/2$ \\
		\hline
		$B_l$ & $\binom{2l+1}{r}$ & $1, r\not=l, (-1)^{l(l+1)/2}, r=l$ \\
		\hline
	           $B_l$ & $2^l$ & $ (-1)^{l(l+1)/2}$  \\
		\hline
		$C_l$ & $\binom{2l}{r}-\binom{2l}{r-2}$ & $ (-1)^r$ \\
		\hline
		$D_l$ & $ \binom{2l}{r}$ & $ 1, 1\leq r\leq l-2 $ \\
		\hline
		$D_l$ & $2^{l-1}$   & $0,r=l-1,\,l\,\text{odd}, (-1)^{l/2}, l\,\text{even} $ \\
		\hline
		$D_l$ & $2^{l-1}$ & $ 0,\,r=l,\,l\,\text{odd}, (-1)^{l/2},l\,\text{even}$ \\
		\hline
 \end{tabular}
 }
 \caption{
 }
 \label{tttab1}
 \end{table}
\end{center}

\begin{center}
 \begin{table}[h]
 \centering
 {\footnotesize
 \begin{tabular}{| c | c | c| }
   \hline
   \multicolumn{3}{|c|}{ \textbf{\textit{Maximal subalgebras of maximal rank in $\mathfrak{g}^c$}}} \\
   \hline                        
   $\mathfrak{g}^c$ & $\mathfrak{\tilde{h}}^c$ & $\operatorname{rank}\mathfrak{\tilde{h}}^c$   \\
   \hline
		$\mathfrak{so}(2l+1,\mathbb{C})$ & $\mathfrak{so}(2k,\mathbb{C})\oplus\mathfrak{so}(2(l-k)+1,\mathbb{C})$ & $l$ \\
		\hline
		$\mathfrak{sp}(2l,\mathbb{C})$ & $\mathfrak{sp}(2k,\mathbb{C}\oplus\mathfrak{sp}(2(l-k)+1,\mathbb{C})$ & $l$ \\
		\hline
	           $\mathfrak{so}(2l,\mathbb{C})$ & $\mathfrak{so}(2k,\mathbb{C})\oplus\mathfrak{so}(2(l-k)+1),\mathbb{C})$ & $ l$  \\
		\hline
		$\mathfrak{sl}(n,\mathbb{C})$ & $\mathfrak{so}(2,\mathbb{C})\oplus\mathfrak{so}(2l-1,\mathbb{C})$ & $ l-1$ \\
		\hline
		$\mathfrak{so}(2l+2,\mathbb{C})$ & $\mathfrak{so}(2,\mathbb{C})\oplus\mathfrak{so}(2l-1),\mathbb{C})$ & $ l-1 $ \\
		\hline
		$\mathfrak{sp}(2l,\mathbb{C})$ & $\mathfrak{gl}(l,\mathbb{C})$   & $l-1$ \\
		\hline
		$\mathfrak{so}(2l,\mathbb{C})$ & $\mathfrak{so}(2,\mathbb{C})\oplus\mathfrak{so}(2l-2,\mathbb{C})$ & $ l-1$ \\
		\hline
                     $\mathfrak{so}(2l,\mathbb{C})$ & $\mathfrak{gl}(l,\mathbb{C})$ & $ l-1$ \\
                     \hline
 \end{tabular}
 }
 \caption{
 }
 \label{tab-max}
 \end{table}
\end{center}


\end{document}